\documentclass[a4paper,11pt,reqno]{amsart}

\usepackage{amsmath,amssymb,amsthm}
\usepackage{enumitem}
\usepackage{xcolor}
\usepackage[margin=30mm]{geometry}
\usepackage[colorlinks=true,linkcolor=blue!55!black,citecolor=blue!55!black,urlcolor=blue!55!black]{hyperref}

\newtheorem{theorem}{Theorem}[section]
\newtheorem{proposition}[theorem]{Proposition}
\newtheorem{lemma}[theorem]{Lemma}
\newtheorem{corollary}[theorem]{Corollary}
\newtheorem{theoremA}{Theorem}

\newtheorem{theoremB}{Theorem}

\theoremstyle{definition}
\newtheorem{question}[theorem]{Question}
\newtheorem{example}[theorem]{Example}
\theoremstyle{remark}
\newtheorem{remark}[theorem]{Remark}

\newcommand{\m}{\mathfrak m}
\newcommand{\calR}{\mathcal R}
\newcommand{\gr}{\operatorname{gr}}
\newcommand{\ord}{\operatorname{ord}}
\newcommand{\Ap}{\operatorname{Ap}}
\newcommand{\ol}[1]{\overline{#1}}
\newcommand{\defword}[1]{\textit{#1}}

\newcommand{\Ker}{\operatorname{Ker}}
\newcommand{\fka}{\mathfrak{a}}
\newcommand{\fkp}{\mathfrak{p}}
\newcommand{\fkm}{\mathfrak{m}}

\newcommand{\Spec}{\operatorname{Spec}}
\newcommand{\embdim}{\operatorname{edim}}
\newcommand{\mult}{\operatorname{mult}}
\newcommand{\Fr}{\operatorname{F}}
\newcommand{\ch}{\operatorname{char}}

\title[Normality of ideals beyond the standard graded setting]
{Normality of ideals beyond the standard graded setting: families from numerical semigroup rings}
\author{Naoyuki Matsuoka}
\date{}
\dedicatory{Dedicated to the memory of Professor Shiro Goto.}
\thanks{The author was partially supported by JSPS KAKENHI Grant Number JP25K06941.}
\thanks{
The author used OpenAI's ChatGPT Work (5.6 Sol) to assist in finding explicit reductions and one-variable homomorphisms used in the coefficient-separation arguments, and in drafting and revising parts of the English exposition. The overall proof strategy was developed by the author. All AI-assisted suggestions and calculations were independently verified by the author, who takes full responsibility for the contents of this paper.
}

\begin{document}

\begin{abstract}
Let $k$ be an arbitrary field and let $S=k[x,y,z]$ be the polynomial ring with three variables $x,y,z$.  We study integrally
closed $(x,y,z)$-primary ideals of $S$ that are homogeneous for a positive
weighted grading but need not be homogeneous for the standard grading.  From
a numerical semigroup $H$ of embedding dimension three, we obtain such ideals
as inverse images
$I_h=\varphi_H^{-1}(t^hk[t]\cap k[H])$.  If $\ell$ is the least degree of a
defining relation of $k[H]$, then $I_\ell$ is monomial, whereas
$I_{\ell+1}$ has a binomial generator in the cases considered here.  We
determine $I_{\ell+1}$ for numerical semigroups of embedding dimension three
and multiplicity three or four.  The six-generated cases arising in
multiplicity three and in the symmetric multiplicity-four case form two
explicit families.  For every ideal $I$ in these families, we prove that its
Rees algebra is a Cohen--Macaulay normal domain.  In the non-symmetric
multiplicity-four case, $I_{\ell+1}$ is seven-generated; for
$H=\langle4,9,15\rangle$, we prove that its Rees algebra is again a
Cohen--Macaulay normal domain.
\end{abstract}

\maketitle
\tableofcontents

\section{Introduction}
\label{sec:introduction}

An ideal is said to be integrally closed if it coincides with its integral
closure, whereas it is said to be normal if all its positive powers are
integrally closed.  Thus an integrally closed ideal need not be normal.  Since
the integral closedness of an ideal alone provides no general procedure for
controlling its higher powers, it is natural to ask what restrictions on the
number and form of the generators ensure that an integrally closed ideal
is normal.  When the ambient ring is a normal domain, the normalization of the
Rees algebra of $I$ is $\bigoplus_{n\geq0}\ol{I^n}T^n$.  Hence normality of
$I$ is equivalent to normality of its Rees algebra, so this question is also
naturally connected with the study of normal blowup algebras.

Results organized by the number of generators provide a natural point of
comparison.  Let $(R,\m)$ be a regular local ring of dimension $d$, let
$I$ be an integrally closed $\m$-primary ideal, and write $\mu_R(I)$ for the
minimal number of generators of $I$.  Goto \cite{Goto1987} proved that the Rees
algebra of $I$ is a Cohen--Macaulay normal domain when $\mu_R(I)=d$, that is,
when $I$ is a complete intersection
\cite[Corollary~1.3]{Goto1987}.  Ciuperc\u{a} \cite{Ciuperca2006} obtained the same conclusion when
$\mu_R(I)=d+1$
\cite[Theorem~1.1]{Ciuperca2006}.  Endo, Goto, Hong, and Ulrich \cite{EndoGotoHongUlrich2026} extended these
results by proving the conclusion whenever $\mu_R(I)\leq d+2$
\cite[Theorem~2.2 and Corollary~2.4]{EndoGotoHongUlrich2026}.  They further proved that, over a field of characteristic zero, the Rees algebra of an
integrally closed zero-dimensional ideal generated by $d+3$ homogeneous
polynomials in the standard graded polynomial ring $k[x_1,x_2,\ldots,x_d]$ is a
Cohen--Macaulay normal domain
\cite[Theorem~3.4]{EndoGotoHongUlrich2026}.  Thus, in three variables, their
results cover every integrally closed zero-dimensional ideal with at most five
minimal generators, and also cover six-generated ideals in the standard
graded case in characteristic zero.

For monomial ideals in three variables, Ataka and the author
\cite{AtakaMatsuoka2026} proved that, over an arbitrary field, every
integrally closed zero-dimensional monomial ideal with at most seven minimal
generators is normal \cite[Theorem~3.1]{AtakaMatsuoka2026}.  
This bound is sharp: the integral closure \(\overline{(x^7,y^3,z^2)}\) is 
minimally generated by eight monomials but is not normal; 
see \cite[Exercise 1.14]{SwansonHuneke2006} and \cite[Example 4.5]{AtakaMatsuoka2026}. 
We revisit this ideal in Example \ref{ex:successive-cutoffs-eight-generators}, where it occurs as the next cutoff after a normal ideal.
Together with the standard facts recalled in the next paragraph, the theorem of
Ataka--Matsuoka also implies that the corresponding Rees algebra is a
Cohen--Macaulay normal domain.  Consequently, among six-generated ideals in
three variables, the first case not covered by the results above is the
non-monomial case outside the standard graded setting; in positive characteristic, even the
standard graded nonmonomial case lies outside the theorem of
Endo--Goto--Hong--Ulrich.

The monomial case has two particularly powerful features.  First, a theorem
of Reid, Roberts, and Vitulli states that a monomial ideal in $d$ variables is
normal once its first $d-1$ positive powers are integrally closed
\cite[Proposition~3.1]{ReidRobertsVitulli2003}.  In particular, for an
integrally closed monomial ideal in three variables, it is enough to verify
the integral closedness of its square.  Second, the Rees algebra of a monomial
ideal is an affine semigroup ring.  Once the ideal is normal, its Rees algebra
is therefore a normal affine semigroup ring and is Cohen--Macaulay by
Hochster's theorem \cite{Hochster1972}.  Neither result applies directly to an
ideal containing a binomial generator.  In particular, the finite criterion
of Reid--Roberts--Vitulli does not reduce the problem studied here to the
integral closedness of the square.  Moreover, when the relevant grading is not
standard, the theorem of Endo--Goto--Hong--Ulrich does not apply either.  Thus,
outside both the monomial and standard graded settings, the preceding results
provide neither a finite criterion for normality nor a direct conclusion that
the Rees algebra is normal.

These results suggest the following question.

\begin{question}\label{que:seven-generators}
Let $k$ be an arbitrary field, let $S=k[x,y,z]$, and let $I$ be an integrally
closed $(x,y,z)$-primary ideal that is homogeneous with respect to some
positive weighted grading of $S$.  If $\mu_S(I)\leq7$, must $I$ be normal?
\end{question}

The purpose of this paper is to study a natural class of six-generated
integrally closed ideals that lies just beyond these monomial and standard
graded settings.  Recall that a numerical semigroup is a submonoid
$H$ of $\mathbb N$ such that $\mathbb N\setminus H$ is finite. Here, $\mathbb{N}$ denotes the set of all non-negative integers.
Its multiplicity $\mult(H)$ is its least positive element, and its embedding
dimension is the cardinality of its minimal system of generators.  Writing
$\Fr(H)=\max(\mathbb N\setminus H)$, called the Frobenius number of $H$, we say that $H$ is symmetric if
$a \in H$ if and only if $\Fr(H)-a \notin H$ for every integer $a$.  We
determine the ideals arising from the numerical semigroup rings $k[H]$ when
$H$ has embedding dimension three and multiplicity three or four.  In the
multiplicity-three and symmetric multiplicity-four cases, apart from the
cases obtained by adjoining one variable to an ideal in a two-variable
polynomial ring, the resulting six-generated ideals form two explicit
families.  We give an affirmative answer to
Question~\ref{que:seven-generators} for these two families by proving that
their Rees algebras are Cohen--Macaulay normal domains.

We first state the main result directly in terms of the resulting ideals.
Their numerical-semigroup origin is explained after Theorem~\ref{thm:main}.  Let $k$ be
an arbitrary field, let $S=k[x,y,z]$, and put $\m=(x,y,z)$.  Consider either
of the following ideals:
\begin{equation*}\tag{$\sharp$}\label{eq:family-one}
\begin{aligned}
 I={}&(y^2-x^{b-d}z,\ x^{2b+2-e},\ x^{b+1}y,\
        x^{b-d+1}z,\ yz,\ z^2),\\
 &\hspace{35mm} b>d\geq0,\qquad e\in\{0,1\},
\end{aligned}
\end{equation*}
or
\begin{equation*}\tag{$\sharp\sharp$}\label{eq:family-two}
\begin{aligned}
 I={}&(y^2-x^{2b+1},\ x^{2b+2},\ x^{b+1}y,\ x^rz,\ yz,\ z^2),\\
 &\hspace{35mm} b\geq1,\qquad 1\leq r\leq b+1.
\end{aligned}
\end{equation*}
Each of these ideals is homogeneous with respect to a suitable positive
grading.  For example, one may take $\deg x=3$, $\deg y=3b+2-e$, and
$\deg z=3b+3d+4-2e$ for \eqref{eq:family-one}, while one may take
$\deg x=4$, $\deg y=4b+2$, and any positive value of $\deg z$ for
\eqref{eq:family-two}.  Each ideal is $\m$-primary and minimally generated
by six elements.  By
Theorem~\ref{thm:classification}(4), every ideal of either form arises from
the numerical-semigroup construction explained after Theorem~\ref{thm:main}
and is therefore integrally closed.

\begin{theoremA}\label{thm:main}
For every ideal $I$ of either form \eqref{eq:family-one} or
\eqref{eq:family-two}, the Rees algebra of $I$ is a Cohen--Macaulay normal
domain.
\end{theoremA}

The ideal in \eqref{eq:family-one} is homogeneous for the standard grading
precisely when $b-d=1$, whereas no ideal in \eqref{eq:family-two} is
homogeneous for the standard grading.  When $b-d=1$, the
characteristic-zero case is covered by the theorem of
Endo--Goto--Hong--Ulrich, but the positive-characteristic case is not.  The
ideals in \eqref{eq:family-one} with $b-d>1$, as well as all the ideals in
\eqref{eq:family-two}, lie outside the scope of that theorem in every
characteristic.  Thus Theorem~\ref{thm:main} goes beyond the ranges covered by the two
results above, including in positive characteristic along the standard graded
boundary.

We now explain the origin of these ideals.  Let
$H=\langle a_1,a_2,a_3\rangle$ be a numerical semigroup of embedding
dimension three, where $a_1,a_2,a_3$ form its minimal system of generators,
let $k[H]=k[t^{a_1},t^{a_2},t^{a_3}]$ be the $k$-subalgebra of the polynomial ring $k[t]$, and consider the graded homomorphism
\[
 \varphi_H\colon S=k[x_1,x_2,x_3]\longrightarrow k[H],
 \qquad x_i\longmapsto t^{a_i},
\]
where $\deg x_i=a_i$.  For each $0<h\in\mathbb N$, set
\[
 I_h=\varphi_H^{-1}\bigl(t^hk[t]\cap k[H]\bigr).
\]
The ideal $t^hk[t]$ is integrally closed in the polynomial ring $k[t]$.
Hence its contraction to $k[H]$, and therefore its inverse image
$I_h$, is integrally closed.

Let $\fkp_H=\Ker\varphi_H$, and let $\ell$ be the least weighted degree of a
nonzero homogeneous element of $\fkp_H$.  If $h\leq\ell$, then $I_h$ is a
monomial ideal.  The descriptions of the defining ideals of numerical
semigroup rings of embedding dimension three used here are based on Herzog's work
\cite{Herzog1970}.  For the numerical semigroups considered in this paper,
the degree-$\ell$ component of $\fkp_H$ is one-dimensional over $k$ and is
spanned by a binomial.  Consequently, when $h=\ell+1$, precisely this relation
remains as a binomial generator, while the relations of larger degree are
absorbed into the monomial part of $I_{\ell+1}$.  This is why the ideals in
Theorem~\ref{thm:main} have five monomial generators and one binomial generator.
The present paper focuses on those $I_{\ell+1}$ that are minimally generated
by six elements, and investigates whether the relations among their
generators provide enough control over arbitrary powers to prove normality.

The following classification shows that the two families in Theorem~\ref{thm:main}
are precisely the six-generated families arising from this construction.
When comparing $I_{\ell+1}$ with the ideals
in \eqref{eq:family-one} and \eqref{eq:family-two}, we use $x,y,z$ for a
suitable ordering of $x_1,x_2,x_3$.

\begin{theoremB}\label{thm:classification}
Let $H$ be a numerical semigroup of embedding dimension three and multiplicity
three or four, let $\ell$ be as above, and consider $I_{\ell+1}$.
\begin{enumerate}[label=\textup{(\arabic*)}]
 \item If $H$ has multiplicity three, then, after a possible interchange of
 variables, $I_{\ell+1}$ is of the form \eqref{eq:family-one}.
 \item If $H$ is symmetric of multiplicity four, equivalently, if $k[H]$ is
 Gorenstein, then, after a possible interchange of variables, either
 \[
  I_{\ell+1}=(x_3)+JS
 \]
 for an integrally closed $(x_1,x_2)$-primary ideal
 $J\subseteq k[x_1,x_2]$, or $I_{\ell+1}$ is of the form
 \eqref{eq:family-one} or \eqref{eq:family-two}.
 \item If $H$ is non-symmetric of multiplicity four, then $I_{\ell+1}$ is
 minimally generated by seven elements.
 \item Conversely, every choice of the parameters in
 \eqref{eq:family-one} is realized by a numerical semigroup of multiplicity
 three, and every choice of the parameters in \eqref{eq:family-two} is
 realized by a symmetric numerical semigroup of multiplicity four.
\end{enumerate}
\end{theoremB}

Thus Theorem~\ref{thm:main} covers every six-generated ideal arising from this construction
when the semigroup has multiplicity three or is symmetric of multiplicity
four, apart from the cases that reduce immediately to two variables.  We
emphasize that Theorem~\ref{thm:classification} classifies the ideals produced by this
numerical-semigroup construction; it is not a classification of all
six-generated integrally closed ideals that are homogeneous for some positive
grading.

These results provide evidence for the expectation that, in a three-dimensional regular local ring $(A,\mathfrak n)$, 
every six-generated integrally closed $\mathfrak n$-primary ideal has a Cohen–Macaulay normal Rees algebra. 
The two families treated here establish this conclusion for a class of non-monomial ideals over an arbitrary field, including ideals that are not homogeneous for the standard grading.

For Cohen–Macaulayness, we explicitly choose a three-generated ideal $Q\subseteq I$ and verify that $I^2=QI+\mathfrak mI^2$. 
After localization, this equality allows us to apply the standard criterion for Cohen–Macaulayness of Rees algebras. 
The verification uses relations among the generators specific to each family.

The first case of Theorem~\ref{thm:classification}(2) is covered by
\cite[Theorem~2.1]{EndoGotoHongUlrich2026}, as explained in
Subsection~\ref{subsec:symmetric-multiplicity-four}.

The non-symmetric case of multiplicity four lies just beyond the scope of
Theorem~\ref{thm:main} because it produces seven generators.  To give a seven-generated
affirmative instance of Question~\ref{que:seven-generators}, we study in
detail the ideal obtained from the numerical semigroup
$H=\langle4,9,15\rangle$.  After determining its seven minimal generators in
the proof of Theorem~\ref{thm:classification}, we prove that it is normal,
that its localization at $\m$ has reduction number one, and that its Rees
algebra is a Cohen--Macaulay normal domain.  This example provides evidence
for the first part of Question~\ref{que:seven-from-semigroups}, concerning
all seven-generated ideals arising from non-symmetric numerical semigroups of
embedding dimension three and multiplicity four.  Although we do not claim a general theorem for this class,
we expect the same method to apply throughout the non-symmetric
multiplicity-four case. Since the additional case analysis is unlikely to
yield further mathematical insight, we restrict the detailed treatment to
one seven-generated example.

The construction also suggests a question about successive degree bounds.
For $h\in H$, set
\[
 \sigma(h)=\min\{n\in H\mid n>h\}.
\]
Assume that both $I_h$ and $I_{\sigma(h)}$ have at most seven minimal
generators.  If $I_h$ is normal, must $I_{\sigma(h)}$ also be normal?
Without the generator bounds, this implication can fail even for two
successive monomial ideals, as Example~\ref{ex:successive-cutoffs-eight-generators}
shows.  The successor
$\sigma(h)$, rather than $h+1$, is the natural quantity here because $h+1$
need not belong to $H$, and several consecutive integers may therefore define
the same ideal.  Although $\ell+1$ need not belong to $H$, one always has
$I_{\ell+1}=I_{\sigma(\ell)}$.  The passage from $I_\ell$ to
$I_{\ell+1}$ is the first passage from a monomial ideal to
an ideal with one binomial generator.  
Because the strategy of Subsection \ref{subsec:normality} uses the normality of $J=I+(f)$ independently of how it is established, it may be applied successively to the pairs $I_{\sigma(h)}\subset I_h$, provided that the product reductions and coefficient-separation conditions can be verified at each step.

The paper is organized as follows.  Section~\ref{sec:general-setup} fixes the
numerical-semigroup setup and explains the two common proof strategies: the
reduction calculation used to prove the Cohen--Macaulayness of the Rees
algebra and the coefficient-separation argument used to prove normality.  In
Section~\ref{sec:target-ideals} we determine the ideals $I_{\ell+1}$ arising
in multiplicities three and four and prove
Theorem~\ref{thm:classification}.
Section~\ref{sec:main-theorem} proves
Theorem~\ref{thm:main}; the families in \eqref{eq:family-one} and
\eqref{eq:family-two} are treated in Subsections~\ref{sec:family-one}
and~\ref{sec:family-two}, respectively.
Section~\ref{sec:nonsymmetric-example} is devoted to the seven-generated ideal
arising from $\langle4,9,15\rangle$.
Section~\ref{sec:further-questions} records several questions suggested by
these results, including the question concerning successive ideals $I_h$.

\section{Setup and strategy of the proofs}
\label{sec:general-setup}

In this section, we recall the basic terminology and describe the strategies used in the proof of the main theorem.
We first review basic notions and properties of numerical semigroups and numerical semigroup rings, and then describe the strategies used to prove the Cohen--Macaulayness and normality of the Rees algebras appearing in the main theorem.

In Subsections~\ref{subsec:cm-strategy} and \ref{subsec:normality}, whenever
an infinite ground field is needed, we may replace $k$ by a purely
transcendental extension and assume that it is infinite; the assertions
concerning Cohen--Macaulayness and integral closedness descend by faithful
flatness.

\subsection{Basic notions of numerical semigroups and numerical semigroup rings}\label{subsec:basicnotion}

A \defword{numerical semigroup} $H$ is a submonoid of $\mathbb{N}$ with finite complement $\mathbb{N} \setminus H$.
For every numerical semigroup $H$, there exists a unique minimal system of generators $\{a_1, a_2, \ldots , a_n\}$ of $H$; that is, $H=\{\lambda_1 a_1 + \lambda_2 a_2 + \cdots + \lambda_n a_n \mid \lambda_1, \lambda_2, \ldots , \lambda_n\in \mathbb{N}\}$, and no proper subset of $\{a_1,a_2,\ldots,a_n\}$ generates $H$. Note that the greatest common divisor of $a_1, a_2, \ldots , a_n$ must be one. 
We write $\left<a_1, a_2, \ldots , a_n\right>$ for the numerical semigroup generated by $a_1, a_2 ,\ldots , a_n$.
The number $n$ is called the \defword{embedding dimension} of $H$, denoted by $\embdim(H)$, and the least positive integer in $H$ is called the \defword{multiplicity} of $H$, denoted by $\mult(H)$.
We have $\embdim(H)\leq\mult(H)$.
For any $h \in H \setminus \{0\}$, we put
$$
\Ap(H,h) = \{h' \in H \mid h' - h \notin H\}
$$
and call it the \defword{Ap\'ery set} of $H$ with respect to $h$. Obviously, if we put
$$
h'_i = \min \{h' \in H \mid h' \equiv i \pmod{h}\}
$$
for each $0 \le i \le h-1$, then $\Ap(H,h) = \{h'_0, h'_1, \ldots , h'_{h-1}\}$.
The largest integer not belonging to $H$ is called the \defword{Frobenius number} of $H$, denoted by $\Fr(H)$.
Hence, we have $a \in H$ for every integer $a \ge \Fr(H) + 1$, and $\Fr(H) = \max \Ap(H,h) - h$ for every $h \in H \setminus \{0\}$. 

Throughout this paper, let $k$ denote an arbitrary field. The semigroup ring of $H = \left<a_1, a_2, \ldots , a_n\right>$ over $k$ is the $k$-subalgebra 
$$
k[H] = k[t^h \mid h \in H] = k[t^{a_1}, t^{a_2}, \ldots , t^{a_n}]
$$
of the polynomial ring $k[t]$. The ring $k[H]$ is a one-dimensional Cohen--Macaulay domain. We regard $k[t]$ as a graded ring by $\deg t =1$, and hence $k[H]$ is a graded subring of $k[t]$.
The embedding dimension of $H$ coincides with the embedding dimension of the ring $k[H]$, and the multiplicity of $H$ is equal to the multiplicity of $k[H]$ with respect to the irrelevant maximal ideal $\fkm = (t^{a_1}, t^{a_2}, \ldots , t^{a_n})$.
Note that, for every $h \in H\setminus \{0\}$, the residue classes of the monomials $t^{h'}$, where $h'\in\Ap(H,h)$, form a basis of the $k$-vector space $k[H]/t^hk[H]$. 
Hence $k[H]$ is Gorenstein if and only if $\Ap(H,h)$ is symmetric in the following sense for every (equivalently, for some) $h\in H \setminus \{0\}$: if we write $\Ap(H,h) = \{c_1 < c_2 < \cdots < c_h\}$, then $c_i + c_{h+1-i} = c_h$ for each $1 \le i \le h$. When this is the case, we call $H$ a \defword{symmetric} numerical semigroup.
Suppose in addition that $H\ne\mathbb N$. If $\embdim(H)=\mult(H)$, equivalently, if the ring $k[H]$ has maximal embedding dimension, then $H$ is symmetric if and only if $\embdim(H) = 2$.

Let $H=\left<a_1, a_2, \ldots ,a_n\right>$, and let $S=k[x_1, x_2, \ldots , x_n]$ be the polynomial ring graded by $\deg x_i =a_i$.
Consider the graded ring homomorphism $\varphi_H: S \to k[H]$ defined by $\varphi_H(x_i) = t^{a_i}$, and put $\fkp_H = \Ker \varphi_H$. We call $\fkp_H$ the \defword{defining ideal} of $k[H]$.

Since only the case $n=3$ will be discussed in the remainder of this paper, the basic results for this case are summarized below. The following results are due to Herzog \cite{Herzog1970}.

\begin{theorem}[\cite{Herzog1970}]\label{thm:herzog}
	Let $H=\left<a_1, a_2, a_3\right>$ be a numerical semigroup with $\embdim(H) = 3$. Then the following assertions hold.
	\begin{enumerate}
		\item $H$ is symmetric if and only if, after a suitable reordering of $a_1, a_2, a_3$, the greatest common divisor $d$ of $a_1$ and $a_2$ is at least 2 and $a_3 \in \left<\frac{a_1}{d}, \frac{a_2}{d}\right>$.
		\item For each $1\le i\le3$, put $q_i = \min\{q >0\mid qa_i \in \left<\{a_1, a_2,a_3\} \setminus \{a_i\}\right>\}$ and choose any $p_{ij}\in\mathbb N$ $(j\ne i)$ such that $q_i a_i = \sum_{j\ne i}p_{ij}a_j$. Then $$\fkp_H = \left(x_i^{q_i} - \prod_{j\ne i}x_j^{p_{ij}} ~\middle|~ i = 1,2,3\right).$$
		If $H$ is non-symmetric, these three binomials form a minimal system of
		generators of $\fkp_H$.
		\item If $H$ is symmetric and $a_1,a_2,a_3$ have been ordered as in (1), then $$\fkp_H = (x_1^{\frac{a_2}{d}} - x_2^{\frac{a_1}{d}}, x_3^d - x_1^{p_1}x_2^{p_2}),$$ where $a_3 = p_1 \frac{a_1}{d} + p_2 \frac{a_2}{d}$ $(p_1, p_2 \in \mathbb{N})$. These two binomials form a minimal system of generators of $\fkp_H$, and this agrees with the presentation in (2).
	\end{enumerate}
\end{theorem}

\begin{example}
	\begin{enumerate}
		\item Let $H=\left<3,5,7\right>$. It is not symmetric because $3,5,7$ are pairwise coprime. In this case, $q_1 = 4$, $q_2 = q_3 = 2$, $p_{12} = p_{13} = p_{21} = p_{23} = p_{32} = 1$, and $p_{31} = 3$. Hence $\fkp_H= (x_1^4 - x_2x_3, x_2^2-x_1x_3,x_3^2-x_1^3x_2)$.
		\item Let $H=\left<4,6,11\right>$. Since $11 = 4 \cdot \frac{4}{2} + \frac{6}{2} \in \left<\frac{4}{2}, \frac{6}{2}\right>$, $H$ is symmetric, and $\fkp_H = (x_1^3-x_2^2, x_3^2-x_1^4x_2)$.
	\end{enumerate}
\end{example}

\begin{corollary}\label{cor:emb4symm}
	Let $H=\left<4,a_2, a_3\right>$ be a numerical semigroup with $\embdim(H) = 3$ and $\mult(H) = 4$.
	Then $H$ is symmetric if and only if one of $a_2$ and $a_3$ is even. When this is the case, the other generator must be odd.
	Moreover, there exist an odd integer $p\geq3$ and an integer $q\geq1$ such
	that $H=\left<4,2p,p+2q\right>$.
\end{corollary}

\begin{proof}
	Suppose that $H$ is symmetric and that both $a_2$ and $a_3$ are odd. 
	Then, by Theorem~\ref{thm:herzog} (1), $d = \gcd(a_2, a_3) \ge 2$ and $4 \in \left<\frac{a_2}{d}, \frac{a_3}{d}\right>$. For each $i=2,3$, $\frac{a_i}{d}$ cannot equal $1$, since otherwise $\embdim(H) < 3$. 
	Since both $\frac{a_2}{d}$ and $\frac{a_3}{d}$ are greater than one, an expression of $4$ as a non-negative integral combination of them forces one of them to be either $2$ or $4$. 
	Hence $\frac{a_i}{d}$ must be even for some $i=2,3$, which implies that $a_i$ is even; this is a contradiction.

	Conversely, suppose that one of $a_2$ and $a_3$ is even.  Denote the even
	generator by $2p$ and the other generator by $a$.  Then $a$ is odd, since
	$\gcd(4,a_2,a_3)=1$.  The integer $p$ is also odd, since otherwise $2p$
	would be a multiple of $4$, contrary to the minimality of the generating
	set.  In particular, $p\geq3$.
	Suppose that $a\leq p$.  Since $a$ and $p$ are odd, we have
	$2p=2a+4(p-a)/2\in\left<4,a\right>$, again contradicting minimality.
	Therefore $a>p$, so $a=p+2q$ for some $q\geq1$.  In particular,
	$a\in\left<2,p\right>$, and Theorem~\ref{thm:herzog}(1) shows that $H$ is
	symmetric.
\end{proof}

\subsection{Strategy for Cohen--Macaulayness}
\label{subsec:cm-strategy}

We next explain the strategy used to prove the Cohen--Macaulayness of the Rees algebra.

For a commutative ring $A$ and an ideal $J$ of $A$, put
$$
\calR_A(J)=A[JT]=\bigoplus_{n\geq0}J^nT^n\subseteq A[T], \qquad \gr_J(A)=\bigoplus_{n\geq0}J^n/J^{n+1}.
$$
where $T$ is an indeterminate over $A$.

Throughout this subsection, let $S=k[x,y,z]$, let $\m=(x,y,z)$, and let $I$ be an $\m$-primary homogeneous ideal of $S$. 
For simplicity, $\calR_S(I)$ will be denoted simply as $\calR(I)$ in what follows.

Our strategy is the standard one of proving that the reduction number is one.  
However, even when $I$ is a monomial ideal, it can be quite difficult to find a minimal reduction of $I$ generated by $3=\dim S$ elements, even in situations where the existence of such a reduction is clear, for example when $k$ is infinite. The ideals considered in this paper are not monomial.
We therefore seek a three-generated ideal whose localization at $\m$ is a reduction of $IS_\m$, and verify directly that the local reduction number is one.

The following proposition shows that this local calculation suffices to prove that $\calR(I)$ is Cohen--Macaulay.  We state it in three variables, which is the setting of this paper, although the same proof works in $n$ variables.

\begin{proposition}
\label{prop:local-reduction-criterion}
Suppose that there is a three-generated ideal $Q\subseteq I$ such that
\[
 I^2=QI+\m I^2.
\]
Then $\calR(I)$ is Cohen--Macaulay.
\end{proposition}

\begin{proof}
Put $A=S_\m$.  The assumed equality
gives
\[
 (IA)^2=(QA)(IA)+(\m A)(IA)^2,
\]
and hence Nakayama's lemma yields $(IA)^2=(QA)(IA)$.  Thus $QA$ is a
reduction of the $\m A$-primary ideal $IA$.  Since $Q$ is generated by three
elements and $\dim A=3$, the ideal $QA$ is a parameter ideal for $A$ and
hence a minimal reduction of $IA$.
By the Valabrega--Valla criterion \cite[Proposition~3.1]{ValabregaValla1978}, the
associated graded ring $\gr_{IA}(A)$ is Cohen--Macaulay.  Since the reduction
number of $IA$ with respect to $QA$ is at most one, it is at most
$\dim A-1=2$.  The Goto--Shimoda criterion
\cite[Theorem~3.1]{GotoShimoda1982} now shows
that $\calR_A(IA)$ is Cohen--Macaulay.

It remains to pass back to $S$.  Let $P\in\Spec\calR(I)$ and put
$\fkp=P\cap S$.  The ring $\calR(I)_P$ is a localization of
$\calR_{S_\fkp}(IS_\fkp)$.  If $\fkp=\m$, $\calR(I)_P$ is a localization of the
Cohen--Macaulay ring $\calR_A(IA)$.  If $\fkp\ne\m$, then
$IS_\fkp=S_\fkp$, since $\sqrt I=\m$.  Hence
\[
 \calR_{S_\fkp}(IS_\fkp)=S_\fkp[T],
\]
which is Cohen--Macaulay.  Thus $\calR(I)$ is Cohen--Macaulay.
\end{proof}

To prove the Cohen--Macaulay part of Theorem~\ref{thm:main}, we give a three-generated ideal $Q\subseteq I$ for each ideal $I$ under consideration.
After writing
\[
 I=Q+(f_1,f_2,\ldots,f_r),
\]
it is enough to check that $f_if_j\in QI+\m I^2$ for every $1\leq i\leq j\leq r$.  These inclusions give $I^2=QI+\m I^2$, and Proposition~\ref{prop:local-reduction-criterion} then shows that $\calR(I)$ is Cohen--Macaulay.

\subsection{Strategy for Normality}
\label{subsec:normality}

We now describe our strategy for proving the normality of the ideals in the families under consideration.

Let $S=k[x_1,\ldots,x_n]$ have a positive weighted grading, let
$\m=(x_1,\ldots,x_n)$, and let $I=(g_1,\ldots,g_\mu)$ be an
$\m$-primary homogeneous ideal with this minimal system of generators;
thus $\mu=\mu_S(I)$. Suppose that there is a homogeneous element
$f$ of degree $\ell$ such that $J=I+(f)$ is normal and $\m f\subseteq I$.
In the applications below, $g_1=f-f'$ is a binomial with monomial terms
$f,f'$, and $J=I+(f)=I+(f')$ is an integrally closed monomial ideal
with at most seven generators. Its normality follows from
\cite{AtakaMatsuoka2026}; the inclusion $\m f\subseteq I$ is checked
for each family. The argument below does not depend on the number of generators of $I$.

Suppose, to the contrary, that $I^s$ is not integrally closed for some $s\geq1$.
Since the integral closure of a homogeneous ideal is again homogeneous, 
we may take a homogeneous element
$\xi\in\ol{I^s}\setminus I^s$
of degree $\delta$. 
Then the normality of $J$ gives
$$
 \xi\in \ol{I^s} \subseteq \ol{J^s} = J^s
 = I^s + fI^{s-1} + \cdots + (f^s).
$$
Now choose the least positive integer $m$ such that
$$
 \xi\in I^s+fI^{s-1}+\cdots+f^m I^{s-m}.
$$
We then have
$$
 \xi=\xi_0+f^m\omega,
$$
for some $\xi_0\in \fka := I^s+fI^{s-1}+\cdots+f^{m-1}I^{s-m+1}$ and $\omega\in I^{s-m}$.
Since $\omega$ is homogeneous of degree $\delta-m\ell$, we may choose homogeneous elements $c_\lambda\in S$, indexed by $\lambda=(\lambda_1,\lambda_2,\ldots,\lambda_{\mu})\in\mathbb N^{\mu}$ with $\lambda_1+\lambda_2+\cdots+\lambda_{\mu}=s-m$.  For such a multi-index, write $g^\lambda=g_1^{\lambda_1}g_2^{\lambda_2}\cdots g_{\mu}^{\lambda_{\mu}}$.  Then
$$
\omega=\sum_{\lambda}c_\lambda g^\lambda,
$$
where
$$
\deg c_\lambda+\lambda_1\deg g_1+\lambda_2\deg g_2+\cdots+\lambda_{\mu}\deg g_{\mu}=\delta-m\ell
$$
whenever $c_\lambda\ne0$.  Since $\m f\subseteq I$, we have
$$
f^m\m I^{s-m}\subseteq f^{m-1}I^{s-m+1}\subseteq\fka.
$$
Thus, after absorbing into $\xi_0$ all terms for which $c_\lambda\in\m$, we may assume that every $c_\lambda$ belongs to $k$.  Consequently, only products of weighted degree $\delta-m\ell$ remain.  For $p,q\in\mathbb N$, put
$$
\begin{aligned}
\Lambda^p_q
=\{\lambda = (\lambda_1,\lambda_2,\ldots,\lambda_{\mu})\in\mathbb N^{\mu}\mid {}&
\lambda_1+\lambda_2+\cdots+\lambda_{\mu}=p,\\
&\lambda_1\deg g_1+\lambda_2\deg g_2+\cdots+\lambda_{\mu}\deg g_{\mu}=q\}.
\end{aligned}
$$
We may now write
$$
\omega=\sum_{\lambda\in\Lambda^{s-m}_{\delta-m\ell}}c_\lambda g^\lambda.
$$

If $fg_i\in I^2$ for some $i$, then, for every $\lambda\in\Lambda^{s-m}_{\delta-m\ell}$ with $\lambda_i>0$, we have $c_\lambda f^m g^\lambda\in\fka$.  By absorbing this term into $\xi_0$, we may therefore replace $c_\lambda$ by zero.
Similarly, if $fg_ig_j\in I^3$ for some $i\ne j$, we may take $c_\lambda=0$ whenever $\lambda_i>0$ and $\lambda_j>0$.  In this way, we restrict the set over which $\lambda$ ranges in each case and denote the resulting subset of $\Lambda^{s-m}_{\delta-m\ell}$ by $\Lambda$.  Thus
$$
\omega = \sum_{\lambda \in \Lambda} c_\lambda g^\lambda.
$$
More generally, suppose that two products $g^\lambda$ and $g^{\lambda'}$
have the same number of factors and the same weighted degree.  If
$g^\lambda-g^{\lambda'}\in\m I^{s-m}$, then
$f^m(g^\lambda-g^{\lambda'})\in\fka$, so either product may be replaced by
the other in $\omega$ without changing $f^m\omega$ modulo $\fka$.
The preceding reductions are not exhaustive.  In some cases, further adjustments of the exponent vectors are made using relations among products of the generators.  All such adjustments, together with the resulting set $\Lambda$, will be specified explicitly in each case.

The remaining step is to separate the coefficients.  We first fix the data
arising from the preceding reduction.  Let $k$ be an infinite field, let
$n,\mu\geq1$ be integers, and let $S=k[x_1,x_2,\ldots,x_n]$.  Let
$I=(g_1,g_2,\ldots,g_\mu)$ be an ideal of $S$, and let $f\in S$.  For
$\lambda=(\lambda_1,\lambda_2,\ldots,\lambda_\mu)\in\mathbb N^\mu$, write
$|\lambda|=\lambda_1+\lambda_2+\cdots+\lambda_\mu$.
Choose integers $s,m$ with $1\leq m\leq s$, a subset
$\Lambda\subseteq\{\lambda\in\mathbb N^\mu\mid |\lambda|=s-m\}$,
an element $\xi_0\in\sum_{j=0}^{m-1}f^jI^{s-j}$, and coefficients
$c_\lambda\in k$ for $\lambda\in\Lambda$.  Put
$$
 \omega=\sum_{\lambda\in\Lambda}c_\lambda g^\lambda,
 \qquad \xi=\xi_0+f^m\omega,
$$
and suppose that $\xi\in\ol{I^s}$.

The essential step is to construct one-variable maps that separate the
remaining products.  We record the resulting criterion.

\begin{lemma}\label{lem:coefficient-separation}
Let $a$ and $\nu$ be positive integers.
For $\varepsilon=(\varepsilon_1,\varepsilon_2,\ldots,\varepsilon_a)
\in(k^\times)^a$ and $\alpha=(\alpha_1,\alpha_2,\ldots,\alpha_a)
\in\mathbb Z^a$, write
$\varepsilon^\alpha=\varepsilon_1^{\alpha_1}\varepsilon_2^{\alpha_2}
\cdots\varepsilon_a^{\alpha_a}$.
For each $\varepsilon\in(k^\times)^a$, suppose that there is a
$k$-algebra homomorphism $\psi_\varepsilon:S\to k[[t]]$ such that
\[
 \ord_t\psi_\varepsilon(f)=\nu-1,\qquad
 \ord_t\psi_\varepsilon(g_i)\geq\nu\quad(1\leq i\leq\mu).
\]
Here, $\ord_t$ denotes the discrete valuation of $k[[t]]$.
Suppose that the set
\[
 G_\varepsilon=\{i\in\{1,\ldots,\mu\}\mid
 \ord_t\psi_\varepsilon(g_i)=\nu\}
\]
is independent of $\varepsilon$, and denote this common set by $G$.
Suppose further that, for each $i\in G$, there are
$u_i\in k^\times$ and $\alpha_i\in\mathbb Z^a$ such that, for every
$\varepsilon\in(k^\times)^a$,
\[
 \psi_\varepsilon(g_i)
 =u_i\varepsilon^{\alpha_i}t^\nu+\text{(higher-order terms)}.
\]
For $\lambda\in\mathbb N^\mu$, write
\[
 \operatorname{supp}\lambda
 =\{i\in\{1,\ldots,\mu\}\mid\lambda_i>0\}.
\]
Put $\Lambda_G=\{\lambda\in\Lambda\mid\operatorname{supp}\lambda\subseteq G\}$.
If $\lambda\mapsto\sum_{i\in G}\lambda_i\alpha_i$ is injective on
$\Lambda_G$, then $c_\lambda=0$ for every $\lambda\in\Lambda_G$.
\end{lemma}

\begin{proof}
Applying $\psi_\varepsilon$ to an equation of integral dependence for
$\xi$ shows that $\psi_\varepsilon(\xi)$ is integral over
$\psi_\varepsilon(I^s)k[[t]]$. Since every ideal of the discrete valuation
ring $k[[t]]$ is integrally closed, we have
\[
 \psi_\varepsilon(\xi)\in\psi_\varepsilon(I^s)k[[t]]
 \subseteq(t^{s\nu}).
\]
On the other hand, since
 $
 \ord_t\psi_\varepsilon(\xi_0)\geq s\nu-m+1.
 $
As $m\geq1$, the equality $\xi=\xi_0+f^m\omega$ therefore gives
\[
 \ord_t\psi_\varepsilon(\omega)\geq(s-m)\nu+1.
\]
The coefficient of $t^{(s-m)\nu}$ in $\psi_\varepsilon(\omega)$ is thus zero:
\[
 \sum_{\lambda\in\Lambda_G}
 c_\lambda\left(\prod_{i\in G}u_i^{\lambda_i}\right)
 \varepsilon^{\sum_{i\in G}\lambda_i\alpha_i}=0.
\]
This holds for every $\varepsilon\in(k^\times)^a$, so it is an identity
of Laurent polynomials over the infinite field $k$. Injectivity makes
the parameter monomials distinct; since each $u_i$ is nonzero, the
asserted coefficients vanish.
\end{proof}

In our applications, $\Lambda$ also has fixed weighted degree
$\sum_i\lambda_i\deg g_i=\delta-m\ell$. The parameter exponents,
together with this equality and $|\lambda|=s-m$, determine $\lambda$;
we verify this injectivity separately in each case.
If one map has $\Lambda_G=\Lambda$, the lemma gives $\omega=0$,
contrary to the minimality of $m$. If several maps are needed, we apply
the lemma successively to the remaining coefficients and check that
the subsets treated by these maps cover $\Lambda$.

\section{Proof of Theorem B: The ideals arising from numerical semigroup rings}
\label{sec:target-ideals}

In this section, we give a proof of Theorem B. 
Throughout the remainder of this paper, let
$H=\langle a_1,a_2,a_3\rangle$ be a numerical semigroup of embedding
dimension three, and retain the notation $S$, $\varphi_H$, and $\fkp_H$
introduced in Subsection~\ref{subsec:basicnotion}.  Put $R=k[H]$, and let $\fkm=(t^{a_1},t^{a_2},t^{a_3})R$ denote the irrelevant maximal ideal of $R$.
For $h\in\mathbb N$, define
$$
 I_h=\varphi_H^{-1}\bigl(t^hk[t]\cap R\bigr).
$$
Equivalently, $I_h$ is generated by $\fkp_H$ together with all monomials of
weighted degree at least $h$.
Let
\[
 \ell=\min\{\deg\rho\mid
 \rho\in\fkp_H\setminus\{0\}\text{ is homogeneous}\}.
\]

Although $I_{\ell+1}$ is defined as an inverse image, this description does
not control its powers.  Indeed,
$\varphi_H^{-1}((t^{\ell+1}k[t]\cap R)^s)=I_{\ell+1}^s+\fkp_H$, which need not
equal $I_{\ell+1}^s$.  In particular, a relation in $\fkp_H$ of degree
$\ell$ belongs to $\varphi_H^{-1}((t^{\ell+1}k[t]\cap R)^2)$ but not to
$I_{\ell+1}^2$, whose nonzero homogeneous elements have degree at least
$2\ell$.  Thus the normality of $I_{\ell+1}$ cannot be deduced merely by
taking inverse images of the powers of $t^{\ell+1}k[t]\cap R$.

\begin{proposition}
\label{prop:inverse-image-ideals}
The following assertions hold.
\begin{enumerate}
	\item The ideal $I_h$ is integrally closed.
	\item If $H$ is symmetric, then $\mu_S(I_h) \le \mult(H) + 2$.
	\item If $H$ is not symmetric, then $\mu_S(I_h) \le \mult(H) + 3$.
	\item $I_h$ is a monomial ideal if and only if $h\leq\ell$.
\end{enumerate}
\end{proposition}

\begin{proof}
(1) The ideal $t^hk[t]$ is integrally closed in $k[t]$.  Its contraction
$t^hk[t]\cap k[H]$ is therefore integrally closed in $k[H]$, and its inverse
image $I_h$ is integrally closed in $S$.

For (2) and (3), the case $h=0$ is immediate, since $I_0=S$. Assume that $h>0$, and put $J=t^hk[t]\cap R$. Then $J$ is an $\fkm$-primary ideal of $R$.  For each $0 \le i \le \mult(H)-1$, choose the least element $h_i$ of $H$ such that $h_i \equiv i \pmod{\mult(H)}$ and $h_i \ge h$.  Then $J=(t^{h_i} \mid 0 \le i \le \mult(H)-1)$. Hence $\mu_R(J)\leq\mult(H)$.
Choose generators $f_1,f_2,\ldots,f_q$ of $J$ with $q\leq\mult(H)$, and choose $g_1,g_2,\ldots,g_q\in S$ such that $\varphi_H(g_i)=f_i$ for every $1\leq i\leq q$.  Since $I_h=\varphi_H^{-1}(J)$, we have
\[
 I_h=\fkp_H+(g_1, g_2, \ldots , g_q).
\]
By Theorem~\ref{thm:herzog}, $\fkp_H$ can be generated by two elements if $H$ is symmetric and by three elements otherwise.  Consequently, $\mu_S(I_h)\leq\mult(H)+2$ in the symmetric case and $\mu_S(I_h)\leq\mult(H)+3$ in the non-symmetric case.

(4) Let $L_h$ be the monomial ideal generated by all monomials of weighted degree
at least $h$.  Then
\[
 I_h=\fkp_H+L_h.
\]
If $h\leq\ell$, every homogeneous element of $\fkp_H$ belongs to $L_h$.
Hence $I_h=L_h$ is a monomial ideal.
Conversely, suppose that $h>\ell$, and choose a nonzero homogeneous binomial in $\fkp_H$ of degree $\ell$.  This binomial belongs to $I_h$, whereas neither of its monomial terms belongs to $I_h$.  Therefore $I_h$ is not a monomial ideal.
\end{proof}

The two types of ideals considered in the main result are both minimally generated by six elements and lie just beyond the monomial case.
In view of Proposition~\ref{prop:inverse-image-ideals} (2) and (3), it is natural to expect that they arise in the following cases.

\begin{itemize}
	\item $H$ has multiplicity three and $h=\ell+1$.
	\item $H$ is symmetric of multiplicity four and $h=\ell+1$.
\end{itemize}

The calculations in the three cases considered in the remainder of this
section show that exactly one defining relation has degree \(\ell\) and that
it remains the unique binomial generator of $I_{\ell+1}$.
For each ideal $I=I_{\ell+1}$ arising in one of these cases, let $f-f'$ denote its unique binomial generator of degree $\ell$.  Since the homogeneous component of $\fkp_H$ of degree $\ell$ is a one-dimensional $k$-vector space spanned by $f-f'$, the monomials $f$ and $f'$ are the only monomials of weighted degree $\ell$.  Passing from $I_{\ell+1}$ to $I_\ell$ amounts precisely to adjoining the degree-$\ell$ monomials $f$ and $f'$.  Since $f-f'\in I_{\ell+1}$, it follows that
$$
I+(f)=I+(f')=I_\ell,
$$
and Proposition~\ref{prop:inverse-image-ideals}(1) shows that this ideal is integrally closed.

We now prove Theorem~\ref{thm:classification}.  In each case, we first determine the Ap\'ery set, the defining ideal $\fkp_H$, and the least relation degree $\ell$.  As in the proof of Proposition~\ref{prop:inverse-image-ideals}, the least elements of $H$ that are at least $\ell+1$ in the residue classes modulo $\mult(H)$ determine the monomial generators of $I_{\ell+1}$ modulo $\fkp_H$.  Combining these monomials with the defining relations gives a generating set of $I_{\ell+1}$.  For the six-generated cases, we then rewrite the parameters to identify the resulting ideals with the two families in the Introduction and give the converse choices of the semigroup parameters.  After treating these two cases, we consider the non-symmetric multiplicity-four case and show that the resulting ideal has seven minimal generators.

Proposition~\ref{prop:inverse-image-ideals} gives only upper bounds for
$\mu_S(I_h)$.  To determine the exact number of generators in the cases
below, we use the following lemma, which separates the contribution of
$\fkp_H$ from that of $t^hk[t]\cap R$.

\begin{lemma}\label{lem:number-of-generators}
For every $h\in\mathbb N$, there is an exact sequence of $k$-vector spaces
\[
0\longrightarrow
\frac{\fkp_H}{\fkp_H\cap(x_1,x_2,x_3)I_h}
\longrightarrow
\frac{I_h}{(x_1,x_2,x_3)I_h}
\longrightarrow
\frac{t^hk[t]\cap R}{\fkm(t^hk[t]\cap R)}
\longrightarrow0.
\]
Consequently,
\[
\mu_S(I_h)
=\dim_k\frac{\fkp_H}{\fkp_H\cap(x_1,x_2,x_3)I_h}
+\mu_R(t^hk[t]\cap R).
\]
In particular, if
\[
\fkp_H\cap(x_1,x_2,x_3)I_h=(x_1,x_2,x_3)\fkp_H,
\]
then
\[
\mu_S(I_h)=\mu_S(\fkp_H)+\mu_R(t^hk[t]\cap R).
\]
\end{lemma}

\begin{proof}
The restriction of $\varphi_H$ to $I_h$ is surjective onto
$t^hk[t]\cap R$ with kernel $\fkp_H$, and
\[
 \varphi_H\bigl((x_1,x_2,x_3)I_h\bigr)
 =\fkm(t^hk[t]\cap R).
\]
Hence $\varphi_H$ induces the asserted exact sequence, and the formulas for
$\mu_S(I_h)$ follow by taking $k$-dimensions.
\end{proof}

In the applications below, the minimal homogeneous generators
$\rho_1,\rho_2,\ldots,\rho_s$ of $\fkp_H$ have pairwise distinct weighted
degrees.  Since $\fkp_H\cap(x_1,x_2,x_3)I_h$ is homogeneous, to verify the
equality
\[
\fkp_H\cap(x_1,x_2,x_3)I_h=(x_1,x_2,x_3)\fkp_H
\]
it is therefore enough to show that
$\rho_i\notin(x_1,x_2,x_3)I_h$ for every $1\leq i\leq s$.

\subsection{Multiplicity three}
\label{subsec:multiplicity-three}
Let
\[
 H=\langle 3,3p+1,3q+2\rangle,\qquad p,q\geq1.
\]
Every numerical semigroup of multiplicity three and embedding dimension
three can be written in this form: after reordering, the two minimal
generators other than $3$ lie in the two nonzero residue classes modulo $3$.
In either of the following membership tests, the coefficient of the generator
not divisible by $3$ must be congruent to $2$ modulo $3$, and its least
possible value is therefore $2$.  It follows that
$3q+2\in\langle3,3p+1\rangle$ if and only if $q\geq2p$, whereas
$3p+1\in\langle3,3q+2\rangle$ if and only if $p\geq2q+1$.  Hence the
integers $3,3p+1,3q+2$ form a minimal system of generators precisely when
\[
 p\leq q\leq2p-1
 \qquad\text{or}\qquad
 q<p\leq2q.
\]
In this case,
\[
 \Ap(H,3)=\{0,3p+1,3q+2\}.
\]

Give $x_1,x_2,x_3$ the weighted degrees $3$, $3p+1$, and $3q+2$, respectively.
The defining ideal is
\[
 \fkp_H=\bigl(x_1^{p+q+1}-x_2x_3,\ x_2^2-x_1^{2p-q}x_3,
                    x_3^2-x_1^{2q-p+1}x_2\bigr).
\]
The degrees of these three relations are
\[
 3(p+q+1),\qquad 6p+2,\qquad 6q+4.
\]
Let $\ell$ be the least of these degrees.

Suppose first that $p\leq q \le 2p-1$.  
Then $\ell = 6p+2$.
Since $\Fr(H)=3q-1 < 6p+3 = \ell + 1$, 
every integer at least $\ell + 1$ belongs to $H$.  Thus the least elements of $H$
at least $\ell + 1$ in the three residue classes modulo $3$ are
\[
 6p+3,\qquad 6p+4,\qquad 6p+5,
\]
represented by
\[
 x_1^{2p+1},\qquad x_1^{p+1}x_2,\qquad x_1^{2p-q+1}x_3,
\]
respectively.  Hence the relation of degree $6p+2$ remains binomial, while
the other two relations contribute only monomial generators to $I_{\ell+1}$, and we
obtain
\[
 I_{\ell+1}=\bigl(x_2^2-x_1^{2p-q}x_3,\ x_1^{2p+1},\ x_1^{p+1}x_2,\
             x_1^{2p-q+1}x_3,\ x_2x_3,\ x_3^2\bigr).
\]

Suppose that $q<p \le 2q$. Then the least relation degree is $\ell = 6q+4$.
In this case $\Fr(H)=3p-2<6q+5 = \ell + 1$, and the same calculation gives
\[
 I_{\ell + 1} =\bigl(x_3^2-x_1^{2q-p+1}x_2,\ x_1^{2q+2},\ x_1^{q+1}x_3,\
             x_1^{2q-p+2}x_2,\ x_2x_3,\ x_2^2\bigr).
\]
In either case, the three monomials obtained from the residue classes have
degrees $\ell+1,\ell+2,\ell+3$ and minimally generate
$t^{\ell+1}k[t]\cap R$.  The three generators of $\fkp_H$ given above have
distinct degrees.  In the first case, the relation of degree $\ell$
does not belong to $(x_1,x_2,x_3)I_{\ell+1}$.  If $q=p$, the other two
relations have degrees $\ell+1$ and $\ell+2$, so neither belongs to this
ideal.  If $q>p$, then, modulo $(x_1,x_2,x_3)I_{\ell+1}$, they are congruent
to $-x_2x_3$ and $x_3^2$, respectively.  Setting $x_1=0$, and then setting
$x_1=x_2=0$, shows that these monomials do not belong to
$(x_1,x_2,x_3)I_{\ell+1}$.  In the second case, the relation of degree
$\ell$ again does not belong to this ideal.  The other two relations are
congruent to $-x_2x_3$ and $x_2^2$, respectively, except that when $p=q+1$
the first of them has degree $\ell+2$.  The same substitutions, with $x_2$
and $x_3$ interchanged for the second monomial, give the required
noncontainments.  Hence
\[
 \fkp_H\cap(x_1,x_2,x_3)I_{\ell+1}=(x_1,x_2,x_3)\fkp_H,
\]
and Lemma~\ref{lem:number-of-generators} gives
\[
 \mu_S(I_{\ell+1})=3+3=6.
\]

To identify these expressions with \eqref{eq:family-one}, we make
the following final choices of variables: $(x,y,z)=(x_1,x_2,x_3)$ in the
first case and $(x,y,z)=(x_1,x_3,x_2)$ in the second case.  With these
choices, both cases are instances of the family
\[
 I=\bigl(y^2-x^{b-d}z,\ x^{2b+2-e},\ x^{b+1}y,\
             x^{b-d+1}z,\ yz,\ z^2\bigr),
 \qquad b>d\geq0,\quad e\in\{0,1\},
\]
where $(b,d,e)=(p,q-p,1)$ for the first case and $(b,d,e)=(q,p-q-1,0)$ for the second case.

Conversely, let $b>d\geq0$ and $e\in\{0,1\}$.  If $e=1$, take
$(p,q)=(b,b+d)$, and if $e=0$, take $(p,q)=(b+d+1,b)$.  In either case,
the required inequalities for $p$ and $q$ hold, and the preceding calculation yields the
ideal in \eqref{eq:family-one}.

The grading inherited from $H$ is given by
$\deg x = 3$, $\deg y = 3b+2-e$, $\deg z = 3b+3d+4-2e$, and hence
$\ell=6b+4-2e$.
These are the weights arising from the numerical-semigroup presentation; the same ideal may admit other positive gradings.

\subsection{Symmetric multiplicity four}
\label{subsec:symmetric-multiplicity-four}
By Corollary~\ref{cor:emb4symm}, $H$ has the form
\[
 H=\langle4,2p,p+2q\rangle,
 \qquad p\geq3\text{ odd},\quad q\geq1.
\]
Its Ap\'ery set with respect to $4$ is
\[
 \Ap(H,4)=\{0,2p,p+2q,3p+2q\}.
\]
Give $x_1,x_2,x_3$ the degrees $4$, $2p$, and $p+2q$, respectively.  The
defining ideal is
\[
 \fkp_H=(x_2^2-x_1^p,\ x_3^2-x_1^qx_2).
\]
The relation degrees are $4p$ and $2p+4q$, so we put $\ell = \min\{4p, 2p+4q\}$.

Assume first that $2q<p$.  Then $\ell=2p+4q$.  Let $w$ be the
integer determined by
\[
 4(w-1)<p+2q+1\leq4w.
\]
The Ap\'ery set of $H$ shows that the least elements at least $\ell + 1$ in
the four residue classes modulo $4$ are represented by
\[
\begin{array}{c|c}
 \text{least element at least }\ell + 1&\text{corresponding monomial}\\ \hline
 2p+4q+2&x_1^{q+\frac{p+1}{2}}\\
 2p+4q+4&x_1^{q+1}x_2\\
 p+2q+4w&x_1^wx_3\\
 3p+2q&x_2x_3.
\end{array}
\]
The last entry is at least $\ell+1=2p+4q+1$ because $2q<p$.
Moreover, since $p$ is odd and $2q<p$, we have
$q+\frac{p+1}{2}\leq p$.
Since
$x_1^p$ is a multiple of $x_1^{q+\frac{p+1}{2}}$, the first relation
of $\fkp_H$ can be replaced by the monomial
$x_2^2$.  We therefore obtain
\[
 I_{\ell+1} =\bigl(x_3^2-x_1^qx_2,\ x_1^{q+\frac{p+1}{2}},\ x_1^wx_3,\ x_1^{q+1}x_2,\ 
             x_2x_3,\ x_2^2\bigr).
\]
To compare this ideal with \eqref{eq:family-one}, we need to rewrite the
exponents in the above generators of $I_{\ell+1}$.
Since $p+2q+1$ is even, exactly one of
$p+2q+1$ and $p+2q+3$ is divisible by $4$.  
Let $e \in \{0,1\}$ such that
\[
 p+2q+1+2e\equiv0\pmod4,
\]
and put
\[
 w=\frac{p+2q+1+2e}{4},\qquad
 b=w-1,\qquad d=w-q-1.
\]
Then $2w-e=q+\frac{p+1}{2}$ and $b>d\geq0$.
Finally, relabelling $(x_1,x_2,x_3)$ as $(x,z,y)$, the ideal becomes the
one in \eqref{eq:family-one}.
Under this relabelling, the grading inherited from $H$ is
$\deg x=4$, $\deg y=p+2q$, and $\deg z=2p$, and the least relation degree is
$\ell=2p+4q$.

Now suppose that $p<2q$.  In this case $\ell =4p$.
If $2q\geq3p+1$, then
$p+2q\geq4p+1=\ell+1$. 
Hence $x_3\in I_{\ell+1}$.  
Put $J=I_{\ell+1}\cap k[x_1,x_2]$.  Then
$I_{\ell+1}=(x_3)+JS$, and $J$ is an integrally closed
$(x_1,x_2)$-primary ideal because it is the contraction of the integrally
closed ideal $I_{\ell+1}$.
In this case, by \cite[Theorem 2.1]{EndoGotoHongUlrich2026}, the Rees algebra of $I_{\ell+1}S_{(x_1,x_2,x_3)}$ is a Cohen--Macaulay normal domain.  Since $I_{\ell+1}S_\fkp=S_\fkp$ for every prime ideal $\fkp\neq(x_1,x_2,x_3)$, the same holds for the Rees algebra of $I_{\ell+1}$.
For this reason, we exclude this case from the family considered in Theorem A. 
We therefore consider the remaining case $p<2q\leq3p$.
Let $r$ be determined by
\[
 4(r-1)+p+2q<4p+1\leq4r+p+2q.
\]
The least elements of $H$ at least $\ell+1$ in the four residue classes modulo $4$ are as follows:
\[
\begin{array}{c|c}
 \text{least element at least }\ell + 1&\text{corresponding monomial}\\ \hline
 4p+4&x_1^{p+1}\\
 4p+2&x_1^{\frac{p+1}{2}}x_2\\
 p+2q+4r&x_1^rx_3\\
 3p+2q&x_2x_3.
\end{array}
\]
Since $2q>p$, the monomial $x_1^qx_2$ is a multiple of
$x_1^{\frac{p+1}{2}}x_2$, so the second relation in
\(\fkp_H\) is replaced by $x_3^2$.  Writing
$p=2b+1$, the inequalities $p < 2q \le 3p$ and
the defining inequality for $r$ give $b\geq1$ and $1\leq r\leq b+1$.
Consequently,
\[
 I_{\ell+1} =\bigl(x_2^2-x_1^{2b+1},\ x_1^{2b+2},\ x_1^{b+1}x_2,\
             x_1^rx_3,\ x_2x_3,\ x_3^2\bigr),
 \qquad b\geq1,\quad1\leq r\leq b+1.
\]

We next verify that the two generating sets obtained above are
minimal.  In each case, the four monomials in the corresponding table
minimally generate $t^{\ell+1}k[t]\cap R$.  Indeed, the degrees of the first
three differ by less than $4$.  A positive difference between the degree of
$x_2x_3$ and either of the first two degrees is an odd integer smaller than
$p+2q$, while its difference from the remaining degree is smaller than $2p$
and congruent to $2$ modulo $4$.  None of these differences belongs to $H$.

The two generators of $\fkp_H$ have distinct degrees.  If $2q<p$, the
relation of degree $\ell$ does not belong to
$(x_1,x_2,x_3)I_{\ell+1}$.  For the other relation, if
$p=2q+1$, its degree is $\ell+2$; otherwise, modulo
$(x_1,x_2,x_3)I_{\ell+1}$, it is congruent to $x_2^2$, which does not belong
to that ideal, as is seen by setting $x_1=x_3=0$.  If $p<2q\leq3p$, the
first relation has degree $\ell$ and the same argument applies to the second:
if $2q=p+1$, its degree is $\ell+2$; otherwise it is congruent modulo
$(x_1,x_2,x_3)I_{\ell+1}$ to $x_3^2$, which does not belong to that ideal,
as is seen by setting $x_1=x_2=0$.
Thus, in both cases,
\[
 \fkp_H\cap(x_1,x_2,x_3)I_{\ell+1}=(x_1,x_2,x_3)\fkp_H.
\]
Lemma~\ref{lem:number-of-generators} now gives
\[
 \mu_S(I_{\ell+1})=2+4=6.
\]
Finally, relabelling $(x_1,x_2,x_3)$ as $(x,y,z)$ gives
an ideal of the form \eqref{eq:family-two}.
The grading inherited from $H$ is given by
$\deg x=4$, $\deg y=2p=4b+2$, and
$\deg z=p+2q$, and the least relation degree is $\ell=4p=8b+4$.
The defining inequality for $r$ shows that $\deg z$ is one of
$8b+5-4r$ and $8b+7-4r$.  The ideal itself is homogeneous for any positive
choice of $\deg z$.

Conversely, let $b\geq1$ and $1\leq r\leq b+1$.  Set
$p=2b+1$ and $q=3b+3-2r$.  Then $p$ is odd, $q\geq1$, and
$p<2q\leq3p$.  Moreover,
\[
 4(r-1)+p+2q<4p+1\leq4r+p+2q,
\]
so the preceding calculation yields the ideal corresponding to the given
pair $(b,r)$.
In Subsection~\ref{sec:family-two}, we use the grading obtained from this
realization; thus $\deg x=4$, $\deg y=4b+2$, and
$\deg z=8b+7-4r$.

For convenience, the parametrizations obtained in the multiplicity-three
and symmetric multiplicity-four cases are summarized below.  
\[
\begin{array}{c|c|c}
 \text{Semigroup}&\text{Range}&\text{Result}\\ \hline\hline
 \langle3,3p+1,3q+2\rangle
   &p\leq q\leq2p-1
   &\eqref{eq:family-one},\quad(b,d,e)=(p,q-p,1)\ \\ \hline
 \langle3,3p+1,3q+2\rangle
   &q<p\leq2q
   &\eqref{eq:family-one},\quad(b,d,e)=(q,p-q-1,0)\ \\ \hline
 \langle4,2p,p+2q\rangle
   &2q<p
   &\eqref{eq:family-one},\quad(b,d,e)=(w-1,w-q-1,e)\ \\ \hline
 \langle4,2p,p+2q\rangle
   &p<2q\leq3p
   &\eqref{eq:family-two},\quad (b,r)=\bigl((p-1)/2,\ \lceil(3p+1-2q)/4\rceil\bigr) \\ \hline
 \langle4,2p,p+2q\rangle
   &2q\geq3p+1
   &I_{\ell+1}=(x_3)+JS
\end{array}
\]

\subsection{Non-symmetric multiplicity four}
\label{subsec:nonsymmetric-multiplicity-four}

Let $H$ be non-symmetric of multiplicity four.  By
Corollary~\ref{cor:emb4symm}, after ordering the generators we may write
\[
 H=\langle4,u,v\rangle,
 \qquad 4<u<v,
\]
where $u$ and $v$ are odd.  The minimality of the generators implies that
$u\not\equiv v\pmod4$ and $v<3u$.  Hence $u+v\equiv0\pmod4$.

Give $x_1,x_2,x_3$ the weighted degrees $4,u,v$, respectively.  By
Theorem~\ref{thm:herzog},
\[
 \fkp_H=\left(
 x_1^{\frac{u+v}{4}}-x_2x_3,
 x_2^3-x_1^{\frac{3u-v}{4}}x_3,
 x_3^2-x_1^{\frac{v-u}{2}}x_2^2
\right).
\]
Moreover,
\[
 \Ap(H,4)=\{0,u,v,2u\}.
\]
The degrees of the three relations are $u+v$, $3u$, and $2v$.  If
$v<2u$, then $u+v$ is the least of these degrees.  If $2u<v<3u$, then
$3u$ is the least.  We therefore treat these two cases separately.

Suppose first that $v<2u$, so that $\ell=u+v$.  Put
\[
 b=\frac{u+v}{4},
\]
and let $e=0$ if $(u,v)\equiv(1,3)\pmod4$ and $e=1$ if
$(u,v)\equiv(3,1)\pmod4$.  Put
\[
 c=\frac{v+2e+1}{4}.
\]
Then
\[
 u=4(b-c)+2e+1,
 \qquad v=4c-2e-1,
\]
and the inequalities $4<u<v<2u$ are equivalent to
\[
 c\geq e+2, \qquad \text{ and }
 \qquad 3c-2e\leq2b\leq4c-2e-2.
\]
Since $\ell+1=4b+1$, the least elements of $H$ that are at least
$\ell+1$ in the four residue classes modulo $4$ are represented as follows:
\[
\begin{array}{c|c|c}
 \text{Ap\'ery element}&\text{least element at least }\ell+1
   &\text{corresponding monomial}\\ \hline
 0&4b+4&x_1^{b+1}\\
 u&4b+2e+1&x_1^cx_2\\
 v&4b-2e+3&x_1^{b-c+1}x_3\\
 2u&4b+2&x_1^{2c-b-e}x_2^2.
\end{array}
\]
The first relation has degree $\ell$ and remains binomial.  Furthermore,
\[
 \frac{3u-v}{4}-(b-c+1)=2b-3c+2e\geq0
\]
and
\[
 \frac{v-u}{2}-(2c-b-e)=2c-b-e-1\geq0.
\]
Thus $x_1^{\frac{3u-v}{4}}x_3$, which appears as the latter monomial term in the second relation, is a multiple
of $x_1^{b-c+1}x_3$, and $x_1^{\frac{v-u}{2}}x_2^2$ is a multiple of $x_1^{2c-b-e}x_2^2$.  Consequently,
\[
 I_{\ell+1}=\bigl(
 x_1^b-x_2x_3,
 x_1^{b+1},
 x_1^cx_2,
 x_1^{2c-b-e}x_2^2,
 x_1^{b-c+1}x_3,
 x_2^3,
 x_3^2
 \bigr).
\]

Suppose next that $2u<v<3u$, so that $\ell=3u$.  Let $e=0$ if
$u\equiv1\pmod4$ and $e=1$ if $u\equiv3\pmod4$, and put
\[
 c=\frac{u-2e-1}{4},
 \qquad b=\frac{3u-v}{4}.
\]
Then
\[
 u=4c+2e+1,
 \qquad v=3u-4b,
\]
and the inequalities $2u<v<3u$ are equivalent to
\[
 c\geq1,
 \qquad 1\leq b\leq c.
\]
Since $\ell+1=3u+1$, the least elements in the four residue classes are
represented as follows:
\[
\begin{array}{c|c|c}
 \text{Ap\'ery element}&\text{least element at least }\ell+1
   &\text{corresponding monomial}\\ \hline
 0&\ell+1+2e&x_1^{3c+2e+1}\\
 u&\ell+2&x_1^{2c+e+1}x_2\\
 2u&\ell+3-2e&x_1^{c+1}x_2^2\\
 v&\ell+4&x_1^{b+1}x_3.
\end{array}
\]
The second relation has degree $\ell$ and remains binomial.  The equalities
\[
 \frac{u+v}{4}-(3c+2e+1)=c-b\geq0
\]
and
\[
 \frac{v-u}{2}-(c+1)=3c+2e-2b\geq0
\]
show that the other two relations may be replaced by $x_2x_3$ and $x_3^2$.
We therefore obtain
\[
 I_{\ell+1}=\bigl(
 x_2^3-x_1^bx_3,
 x_1^{3c+2e+1},
 x_1^{2c+e+1}x_2,
 x_1^{c+1}x_2^2,
 x_1^{b+1}x_3,
 x_2x_3,
 x_3^2
 \bigr).
\]

In both cases, the four monomials obtained from the four residue classes have
weighted degrees $\ell+1,\ell+2,\ell+3,\ell+4$, in some order.  Their images
therefore form a minimal system of generators of
$t^{\ell+1}k[t]\cap R$, since the least positive element of $H$ is $4$.
The generating sets of $I_{\ell+1}$ obtained above also satisfy
\[
 \fkp_H\cap(x_1,x_2,x_3)I_{\ell+1}
 =(x_1,x_2,x_3)\fkp_H.
\]
Indeed, the relation of degree $\ell$ cannot belong to
$(x_1,x_2,x_3)I_{\ell+1}$, whose nonzero homogeneous elements have degree at
least $\ell+4$.  In the case $v<2u$, if
either $2b-3c+2e$ or $2c-b-e-1$ is positive, the corresponding relation is
congruent modulo $(x_1,x_2,x_3)I_{\ell+1}$ to $x_2^3$ or $x_3^2$,
respectively.  If the difference is zero, that relation has degree at most
$\ell+3$.  The substitutions
$x_1=x_3=0$ and $x_1=x_2=0$ show that neither $x_2^3$ nor $x_3^2$ belongs to
$(x_1,x_2,x_3)I_{\ell+1}$.  In the case $2u<v<3u$, the third relation is
congruent to $x_3^2$.  If $c>b$, the first relation is congruent to
$-x_2x_3$; if $c=b$, it has degree at most $\ell+3$.  After setting $x_1=0$,
the monomial $x_2x_3$ does not belong to
$(x_2,x_3)(x_2^3,x_2x_3,x_3^2)$, and the assertion for $x_3^2$ follows by
setting $x_1=x_2=0$.  Since the three relation degrees are distinct,
this proves the claim.
Lemma~\ref{lem:number-of-generators} now gives
\[
 \mu_S(I_{\ell+1})=3+4=7.
\]
This proves Theorem~\ref{thm:classification}(3), and the converse constructions in the preceding two
subsections prove Theorem~\ref{thm:classification}(4).  The proof of Theorem~\ref{thm:classification} is complete.

\section{Proof of Theorem A}
\label{sec:main-theorem}

\subsection{The first six-generated family}
\label{sec:family-one}

In this subsection, we treat the first family \eqref{eq:family-one} in the Introduction. Put
$$
\begin{aligned}
 I={}&( g_1 = y^2-x^{b-d}z,\ g_2 = x^{2b+2-e},\ g_3 = x^{b+1}y,\
        g_4 = x^{b-d+1}z,\ g_5 = yz,\ g_6 = z^2),\\
 &\hspace{35mm} b>d\geq0,\qquad e\in\{0,1\}.
\end{aligned}
$$

\subsubsection{Cohen--Macaulayness}

\begin{proposition}
\label{prop:family-one-local-reductions}
Define $Q\subseteq I$ by
\[
\begin{array}{c|c|l}
 e&d&Q\\ \hline
 1&0,1&(g_1,g_5,g_2+g_6)\\
 1&d\geq2&(g_1,g_2+g_5,g_4+g_6)\\
 0&0&(g_1+g_5,g_3+g_5,g_2+g_4+g_5+g_6)\\
 0&d\geq1&(g_1,g_3+g_5,g_2+g_4+g_6).
\end{array}
\]
Then
\[
 I^2=QI+\m I^2.
\]
Consequently, $\calR(I)$ is Cohen--Macaulay by Proposition~\ref{prop:local-reduction-criterion}.
\end{proposition}

\begin{proof}
	We consider separately the four cases in the definition of $Q$.
	
	\medskip
	
	\noindent
	\underline{$e=1$, $d =0, 1$}: Since $Q=(g_1, g_5, g_2+g_6)$, we have $I=Q+ (g_3, g_4, g_6)$.
	Hence $I^2=QI + (g_ig_j \mid i,j \in \{3,4,6\})$.
	Since $d\in\{0,1\}$, all the powers of $x$ occurring below are non-negative.
	\begin{itemize}
		\item $g_4g_6=-xg_1g_6+xg_5^2\in QI+\m I^2$.
		\item $g_3^2=xg_1g_2+g_4(g_2+g_6)-g_4g_6\in QI+\m I^2$.
		\item $
		 g_3g_4=x^{1-d}g_5(g_2+g_6)-x^{1-d}g_5g_6\in QI.
		$
		\item $
		 g_3g_6=x^dg_4g_5\in QI.
		$
		\item $
		 g_4^2=-xg_1g_4+x^{1-d}g_3g_5\in QI.
		$
		\item $
		 g_2g_6=-x^{2d}g_1g_4+x^dg_3g_5\in QI,
		$
		we obtain
		$
		 g_6^2=(g_2+g_6)g_6-g_2g_6\in QI.
		$
	\end{itemize}
	Thus all the six products belong to $QI+\m I^2$, and hence
	$I^2=QI+\m I^2$.

	\medskip
	
	\noindent
	\underline{$e=1$, $d\geq2$}: Since $Q=(g_1,g_2+g_5,g_4+g_6)$, we have $I=Q+(g_3,g_5,g_6)$.
	Hence $I^2=QI+(g_ig_j\mid i,j\in\{3,5,6\})$.
	Since $d\geq2$, all the powers of $x$ occurring below are positive.
	\begin{itemize}
		\item $g_3^2=xg_1g_2+(g_4+g_6)g_2-x^{2d-1}g_4^2\in QI+\m I^2$.
		\item $g_3g_5=x^dg_1g_4+x^{d-1}g_4^2\in QI+\m I^2$.
		\item $g_3g_6=x^dg_4g_5\in\m I^2$.
		\item $g_5^2=(g_2+g_5)g_5-x^{d-1}g_3g_4\in QI+\m I^2$.
		\item $g_5g_6=(g_2+g_5)g_6-x^{2d-1}g_4^2\in QI+\m I^2$.
		\item $g_6^2=(g_4+g_6)g_6-x(g_5^2-g_1g_6)\in QI+\m I^2$.
	\end{itemize}
	Thus all the six products belong to $QI+\m I^2$, and hence
	$I^2=QI+\m I^2$.

	\medskip
	
	\noindent
	\underline{$e=0$, $d=0$}: Since
	$Q=(g_1+g_5,g_3+g_5,g_2+g_4+g_5+g_6)$, we have
	$I=Q+(g_2,g_4,g_5)$.
	Hence $I^2=QI+(g_ig_j\mid i,j\in\{2,4,5\})$.
	We first note that
	$g_2g_6=-xg_1g_4+xg_3g_5\in\m I^2$ and
	$g_4g_6=-xg_1g_6+xg_5^2\in\m I^2$.
	\begin{itemize}
		\item $g_2g_4=x(g_3^2-g_1g_2)\in\m I^2$.
		\item $g_4^2=g_2g_6\in\m I^2$.
		\item $g_4g_5=(g_2+g_4+g_5+g_6)g_4-g_2g_4-g_4^2-g_4g_6\in QI+\m I^2$.
		\item Since $g_3g_4=g_2g_5$, we have
		$g_2g_5=(g_3+g_5)g_4-g_4g_5\in QI+\m I^2$.
		\item Since $g_3g_6=g_4g_5$, we have
		$g_5^2=(g_2+g_4+g_5+g_6)g_5-g_2g_5-(g_3+g_5)g_6\in QI+\m I^2$.
		\item $g_2^2=(g_2+g_4+g_5+g_6)g_2-g_2g_4-g_2g_5-g_2g_6\in QI+\m I^2$.
	\end{itemize}
	Thus all the six products belong to $QI+\m I^2$, and hence
	$I^2=QI+\m I^2$.

	\medskip
	
	\noindent
	\underline{$e=0$, $d\geq1$}: Since $Q=(g_1,g_3+g_5,g_2+g_4+g_6)$, we have
	$I=Q+(g_2,g_4,g_5)$.
	Hence $I^2=QI+(g_ig_j\mid i,j\in\{2,4,5\})$.
	We first note that
	$g_2g_6=-x^{2d+1}g_1g_4+x^{d+1}g_3g_5\in\m I^2$ and
	$g_4g_6=-xg_1g_6+xg_5^2\in\m I^2$.
	\begin{itemize}
		\item $g_2g_4=x(g_3^2-g_1g_2)\in\m I^2$.
		\item $g_2g_5=x^dg_3g_4\in\m I^2$.
		\item $g_4^2=(g_2+g_4+g_6)g_4-g_2g_4-g_4g_6\in QI+\m I^2$.
		\item Since $g_3g_6=x^dg_4g_5\in\m I^2$, we have
		$g_5g_6=(g_3+g_5)g_6-g_3g_6\in QI+\m I^2$, and hence
		$g_4g_5=(g_2+g_4+g_6)g_5-g_2g_5-g_5g_6\in QI+\m I^2$.
		\item $g_5^2=(g_3+g_5)g_5-x^dg_1g_4-x^{d-1}g_4^2\in QI+\m I^2$.
		\item $g_2^2=(g_2+g_4+g_6)g_2-g_2g_4-g_2g_6\in QI+\m I^2$.
	\end{itemize}
	Thus all the six products belong to $QI+\m I^2$, and hence
	$I^2=QI+\m I^2$.
\end{proof}

\begin{remark}
Proposition~\ref{prop:family-one-local-reductions} shows that $QS_\m$ is a
reduction of $IS_\m$; it does not assert that $Q$ itself is a reduction of
$I$ in $S$.  Since no global reduction is needed for the proof of
Theorem~\ref{thm:main}, we do not pursue one here.
\end{remark}

\subsubsection{Normality}

We next prove that $I$ is normal.
Recall from Subsection~\ref{subsec:multiplicity-three} that
$\deg x=3$, $\deg y=3b+2-e$, and $\deg z=3b+3d+4-2e$; in particular,
the least relation degree is $\ell=6b+4-2e$.

Let $f=y^2$ and $J=I+(f)$. Since $\varphi_H(f) = t^{6b+4-2e} = t^\ell$, 
$$
J=(y^2, x^{b-d}z, x^{2b+2-e}, 	x^{b+1}y, yz, z^2)
$$
is a monomial ideal with six minimal generators.
As observed in Section~\ref{sec:target-ideals}, $J=I_\ell$ and hence $J$ is integrally closed.  Since $J$ is a monomial ideal with six minimal generators, it is normal by \cite{AtakaMatsuoka2026}.

\begin{lemma}\label{lem:common_firstfamily}
The following assertions hold.
\begin{enumerate}
	\item $\m f \subseteq I$.
	\item $g_6 f \in I^2$.
\end{enumerate}
	
\end{lemma}

\begin{proof}
	(1) $y^2z = y(yz) \in I$ is obvious.
	Furthermore, $xy^2 = xg_1 + g_4$ and $y^3 = yg_1 + x^{b-d}(yz)$ imply $xy^2, y^3 \in I$.
	
	(2) It is easy that $g_6 f = y^2z^2 = (yz)^2 \in I^2$.
\end{proof}

We now apply the strategy described in Subsection~\ref{subsec:normality}.

Suppose that $I^s$ is not integrally closed for some positive integer $s$.
Since $\ol{I^s}$ is homogeneous, take a homogeneous element
$\xi\in\ol{I^s}\setminus I^s$ of degree $\delta$.  Since
$\xi\in\ol{I^s}\subseteq\ol{J^s}=J^s$, we may choose the least positive
integer $m$ such that
$$
\xi \in I^s + fI^{s-1} + \cdots + f^m I^{s-m}
$$
and write
$$
\xi = \xi_0 + f^m \omega
$$
with $\xi_0\in I^s+fI^{s-1}+\cdots+f^{m-1}I^{s-m+1}$ and
$\omega\in I^{s-m}$.  We may assume that $\omega$ is homogeneous of degree
$\delta-m\ell$.  By Lemma~\ref{lem:common_firstfamily}(2) and the reductions
in Subsection~\ref{subsec:normality}, we may put
$$
\Lambda_0=\{\lambda\in\Lambda^{s-m}_{\delta-m\ell}\mid\lambda_6=0\}
$$
and write
$$
\omega = \sum_{\lambda \in \Lambda_0} c_\lambda g^\lambda
$$
with $c_\lambda\in k$.

\begin{lemma}\label{lem:firstfamily_cases}
	The following assertions hold.
	\begin{enumerate}
		\item If $e=0$, then $fg_2 \in I^2$.
		\item Suppose $d=0$. Then $fg_4 \in I^2$. Furthermore, if $e=1$, then $fg_3 \in I^2$.
		\item Suppose $d=1$. Then $g_3g_5 \equiv g_4^2 \pmod{\m I^2}$. Furthermore, if $e=1$, then $g_3^2 \equiv g_2g_4 \pmod{\m I^2}$, $g_2g_5 = g_3g_4$, and
		$g_3^3\equiv g_2^2g_5\pmod{\m I^3}$.
		\item Suppose $d\geq2$. Then $g_3g_5\in\m I^2$.  If $e=1$, then $g_2g_5\in\m I^2$ and $g_3^2\equiv g_2g_4\pmod{\m I^2}$.
	\end{enumerate}
\end{lemma}

\begin{proof}
The assertions follow from the following identities:
\begin{align*}
 fg_2&=g_3^2 &&(e=0),\\
 fg_4&=g_3g_5 &&(d=0),\\
 fg_3&=g_1g_3+g_2g_5 &&(d=0,\ e=1),\\
 g_4^2-g_3g_5&=-xg_1g_4 &&(d=1),\\
 g_3^2-g_2g_4&=xg_1g_2 &&(d=1,\ e=1),\\
 g_2g_5&=g_3g_4 &&(d=1,\ e=1),\\
 g_3^3-g_2^2g_5&=xg_1g_2g_3 &&(d=1,\ e=1),\\
 g_3g_5&=x^dg_1g_4+x^{d-1}g_4^2 &&(d\geq2),\\
 g_2g_5&=x^{d-1}g_3g_4 &&(d\geq2,\ e=1),\\
 g_3^2-g_2g_4&=xg_1g_2 &&(d\geq2,\ e=1).
\end{align*}
\end{proof}

Using Lemma~\ref{lem:firstfamily_cases}, we can impose the
following additional conditions on the elements of $\Lambda_0$.  In each
case, let $\Lambda$ denote the resulting set.  We may then write
$$\omega=\sum_{\lambda\in\Lambda}c_\lambda g^\lambda.$$

For $d=0$, the products involving $f$ in Lemma~\ref{lem:firstfamily_cases}
give the first two rows below.  Suppose that $d=1$ and $e=1$.  We first use
$g_4^2\equiv g_3g_5$ to arrange $0\leq\lambda_4\leq1$.  If $\lambda_4=1$
and either $\lambda_2>0$ or $\lambda_3>0$, the other two quadratic relations remove
$g_4$.  When $\lambda_4=0$, the cubic relation allows us to arrange
$0\leq\lambda_3\leq2$.  If $d=1$ and $e=0$, the equality $fg_2=g_3^2$
removes $g_2$, and $g_4^2\equiv g_3g_5$ again gives
$0\leq\lambda_4\leq1$.  Finally, suppose that $d\geq2$.  If $e=1$, use
$g_3^2\equiv g_2g_4$ to arrange $0\leq\lambda_3\leq1$; when $\lambda_5>0$,
the two products belonging to $\m I^2$ remove $g_2$ and $g_3$.  If $e=0$,
the equality $fg_2=g_3^2$ removes $g_2$, and $g_3g_5\in\m I^2$ removes
$g_3$ whenever $\lambda_5>0$.

\[
\begin{array}{c|l}
 \text{Case}&\text{Additional Conditions} \\ \hline \hline
 d=0,\, e=1 & \lambda_3 = \lambda_4 = 0\\
 \hline
 d=0,\, e=0 & \lambda_2 = \lambda_4 = 0\\
 \hline
 d=1,\, e=1 & \begin{gathered}0\leq\lambda_3\leq2,\quad 0\leq\lambda_4\leq1,\\
                    \text{if }\lambda_4=1,\text{ then }\lambda_2=\lambda_3=0
                    \end{gathered}\\
 \hline
 d=1,\, e=0 & \lambda_2 = 0, \, 0 \le \lambda_4 \le 1\\
 \hline
 d\ge 2, \, e=1 & \begin{gathered}0\leq\lambda_3\leq1,\\
                    \text{if }\lambda_5>0,\text{ then }\lambda_2=\lambda_3=0
                    \end{gathered}\\
 \hline
 d\ge 2, \, e=0 & \begin{gathered}\lambda_2=0,\\
                    \text{if }\lambda_5>0,\text{ then }\lambda_3=0
                    \end{gathered}
\end{array}
\]

We now prove, in each of the six cases, that
$c_\lambda=0$ for every $\lambda\in\Lambda$.  For a homomorphism
$\psi_\varepsilon\colon S\to k[[t]]$, we write
$\ord_\varepsilon(g)=\ord_t\psi_\varepsilon(g)$ for $g\in S$.

\medskip

\noindent
\underline{$d=0$, $e=1$}: 
For each $\varepsilon \in k^{\times}$, let $\psi_\varepsilon : S \to k[[t]]$ be a ring homomorphism defined by
$$
x \mapsto t^3,\quad y \mapsto \varepsilon t^{3b+1}, \quad z \mapsto \varepsilon^2 t^{3b+2} - t^{3b+3}.
$$
We note the following:
\[
\begin{array}{c|c|c}
 \text{Element}&\psi_\varepsilon&\ord_\varepsilon\\ \hline
 f=y^2&\varepsilon^2t^{6b+2}&6b+2\\
 g_1=y^2-x^bz&t^{6b+3}&6b+3\\
 g_2=x^{2b+1}&t^{6b+3}&6b+3\\
 g_5=yz&\varepsilon^3t^{6b+3}-\varepsilon t^{6b+4}&6b+3
\end{array}
\]

Put $\nu=6b+3$. The omitted generators $g_3,g_4,g_6$ have
orders $\nu+1,\nu+2,\nu+1$, respectively. Hence
$\psi_\varepsilon(I)k[[t]]\subseteq(t^\nu)$. Set
\[
 G=\{i\mid\ord_t\psi_\varepsilon(g_i)=\nu\}=\{1,2,5\}.
\]
Every $\lambda\in\Lambda$ has support in $G$.
By Lemma~\ref{lem:coefficient-separation}, it suffices to check that
$\lambda\mapsto3\lambda_5$ is injective on $\Lambda$.

For a fixed integer $i$, an element
$\lambda\in\Lambda$ with $\lambda_5=i$ is determined by the following two
equations in $\lambda_1$ and $\lambda_2$:
$$
\begin{cases}
	\lambda_1 + \lambda_2 = s-m-i, \\
	\lambda_1 \deg g_1 + \lambda_2 \deg g_2
	= \delta-m\ell-i\deg g_5.
\end{cases}
$$
Since $\deg g_2-\deg g_1=1$, this system has at most one
solution. Thus the required map is injective and
Lemma~\ref{lem:coefficient-separation} gives $c_\lambda=0$ for every
$\lambda\in\Lambda$.

\medskip

\noindent
\underline{$d=0$, $e=0$}: 
For each $\varepsilon\in k^\times$, define
$\psi_\varepsilon\colon S\to k[[t]]$ by
$$
x\longmapsto t^2,\qquad
y\longmapsto t^{2b+1},\qquad
z\longmapsto t^{2b+2}-\varepsilon t^{2b+3}.
$$
Put $\nu=4b+3$.  The elements that occur in $\omega$ have the following
images:
\[
\begin{array}{c|c|c}
 \text{Element}&\psi_\varepsilon&\ord_\varepsilon\\ \hline
 f=y^2&t^{4b+2}&\nu-1\\
 g_1=y^2-x^bz&\varepsilon t^{4b+3}&\nu\\
 g_3=x^{b+1}y&t^{4b+3}&\nu\\
 g_5=yz&t^{4b+3}-\varepsilon t^{4b+4}&\nu
\end{array}
\]
The omitted generators $g_2,g_4,g_6$ all have order $\nu+1$.
Thus $\psi_\varepsilon(I)k[[t]]\subseteq(t^\nu)$. Set
\[
 G=\{i\mid\ord_t\psi_\varepsilon(g_i)=\nu\}=\{1,3,5\}.
\]
Every $\lambda\in\Lambda$ has support in $G$. Lemma~\ref{lem:coefficient-separation} reduces the claim to
injectivity of $\lambda\mapsto\lambda_1$ on $\Lambda$.

Fix $i$.  An element $\lambda\in\Lambda$ with $\lambda_1=i$ is determined
by
$$
\begin{cases}
 \lambda_3+\lambda_5=s-m-i,\\
 \lambda_3\deg g_3+\lambda_5\deg g_5
 =\delta-m\ell-i\deg g_1.
\end{cases}
$$
Since $\deg g_5-\deg g_3=1$, this system has at most one solution.
Therefore all the coefficients $c_\lambda$ vanish.

\medskip

\noindent
\underline{$d=1$, $e=1$}:
For $\varepsilon\in k^\times$, let $\psi_\varepsilon\colon S\to k[[t]]$
be given by
$$
x\longmapsto t,\qquad
y\longmapsto t^b,\qquad
z\longmapsto t^{b+1}-\varepsilon t^{b+2}.
$$
With $\nu=2b+1$, we have
\[
\begin{array}{c|c|c}
 \text{Element}&\psi_\varepsilon&\ord_\varepsilon\\ \hline
 f=y^2&t^{\nu-1}&\nu-1\\
 g_1&\varepsilon t^\nu&\nu\\
 g_2&t^\nu&\nu\\
 g_3&t^\nu&\nu\\
 g_4&t^\nu-\varepsilon t^{\nu+1}&\nu\\
 g_5&t^\nu-\varepsilon t^{\nu+1}&\nu
\end{array}
\]
The remaining generator $g_6$ has order $\nu+1$, so
$\psi_\varepsilon(I)k[[t]]\subseteq(t^\nu)$. Set
\[
 G=\{i\mid\ord_t\psi_\varepsilon(g_i)=\nu\}=\{1,2,3,4,5\}.
\]
Every $\lambda\in\Lambda$ has support in $G$.
We check the injectivity of $\lambda\mapsto\lambda_1$ on $\Lambda$
to apply Lemma~\ref{lem:coefficient-separation}.

First suppose that $\lambda_4=0$.  After fixing $\lambda_1=i$, the remaining
exponents satisfy
$$
\lambda_2+\lambda_3+\lambda_5=s-m-i
$$
and
$$
\lambda_3+3\lambda_5
=\delta-m\ell-i\deg g_1-(6b+3)(s-m-i).
$$
The right-hand side is fixed.  Since $0\leq\lambda_3\leq2$, its residue
modulo $3$ determines $\lambda_3$, after which $\lambda_5$ and $\lambda_2$
are determined in this order.  Thus there is at most one such element of
$\Lambda$.

If $\lambda_4=1$, then $\lambda_2=\lambda_3=0$, and the number of factors
determines $\lambda_5$.  It remains to see that the two forms cannot occur
with the same value of $i$ and the same weighted degree.  Suppose that
$(i,\lambda_2,\lambda_3,0,\lambda_5,0)$ and
$(i,0,0,1,\mu_5,0)$ did so.  The number of factors gives
$\mu_5=\lambda_2+\lambda_3+\lambda_5-1$.  Comparing their weighted degrees
then yields
$$
3\lambda_2+2\lambda_3=1,
$$
which has no solution in non-negative integers. Hence the required
map is injective, and Lemma~\ref{lem:coefficient-separation} gives
$c_\lambda=0$ for every $\lambda\in\Lambda$.

\medskip

\noindent
\underline{$d=1$, $e=0$}:
Use the same map as in the preceding case, again with $\nu=2b+1$.  For the
elements occurring in $\omega$, we have
\[
\begin{array}{c|c|c}
 \text{Element}&\psi_\varepsilon&\ord_\varepsilon\\ \hline
 f=y^2&t^{\nu-1}&\nu-1\\
 g_1&\varepsilon t^\nu&\nu\\
 g_3&t^\nu&\nu\\
 g_4&t^\nu-\varepsilon t^{\nu+1}&\nu\\
 g_5&t^\nu-\varepsilon t^{\nu+1}&\nu
\end{array}
\]
The omitted generators $g_2,g_6$ have order $\nu+1$. Hence
$\psi_\varepsilon(I)k[[t]]\subseteq(t^\nu)$. Set
\[
 G=\{i\mid\ord_t\psi_\varepsilon(g_i)=\nu\}=\{1,3,4,5\}.
\]
Every $\lambda\in\Lambda$ has support in $G$.
We apply Lemma~\ref{lem:coefficient-separation} by checking
injectivity of $\lambda\mapsto\lambda_1$ on $\Lambda$.

Fix $\lambda_1=i$.  Since $\lambda_2=\lambda_6=0$, the number of factors is
$$
\lambda_3+\lambda_4+\lambda_5=s-m-i.
$$
The weighted-degree equation, after subtracting
$(6b+5)(s-m-i)$ and dividing by $2$, has the form
$$
\lambda_4+2\lambda_5=c
$$
for a fixed integer $c$.  Since $0\leq\lambda_4\leq1$, the parity of $c$
determines $\lambda_4$, and then $\lambda_5$ and $\lambda_3$ are uniquely
determined.  Thus every coefficient $c_\lambda$ is zero.

\medskip

\noindent
\underline{$d\geq2$, $e=1$}:
We first use the homomorphism $\psi_\varepsilon\colon S\to k[[t]]$ defined by
$$
x\longmapsto t,\qquad
y\longmapsto t^b,\qquad
z\longmapsto t^{b+d}-\varepsilon t^{b+d+1}.
$$
Put $\nu_1=2b+1$.  Then
\[
\begin{array}{c|c|c}
 \text{Element}&\psi_\varepsilon&\ord_\varepsilon\\ \hline
 f=y^2&t^{\nu_1-1}&\nu_1-1\\
 g_1&\varepsilon t^{\nu_1}&\nu_1\\
 g_2&t^{\nu_1}&\nu_1\\
 g_3&t^{\nu_1}&\nu_1\\
 g_4&t^{\nu_1}-\varepsilon t^{\nu_1+1}&\nu_1\\
 g_5&t^{\nu_1+d-1}-\varepsilon t^{\nu_1+d}&\nu_1+d-1
\end{array}
\]
The remaining generator $g_6$ has order $2b+2d>\nu_1$.
Thus $\psi_\varepsilon(I)k[[t]]\subseteq(t^{\nu_1})$. Set
\[
 G_1=\{i\mid\ord_t\psi_\varepsilon(g_i)=\nu_1\}=\{1,2,3,4\}.
\]
The elements of $\Lambda$ with support in $G_1$ are precisely those with $\lambda_5=0$. For this subset,
Lemma~\ref{lem:coefficient-separation} requires injectivity of
$\lambda\mapsto\lambda_1$.

After fixing $\lambda_1=i$, the remaining exponents satisfy
$$
\lambda_2+\lambda_3+\lambda_4=s-m-i
$$
and
$$
\lambda_3+2\lambda_4
=\delta-m\ell-i\deg g_1-(6b+3)(s-m-i).
$$
Since $0\leq\lambda_3\leq1$, the parity of the right-hand side determines
$\lambda_3$, and the two equations then determine $\lambda_4$ and
$\lambda_2$.  Hence the coefficients with $\lambda_5=0$ vanish.

It remains to treat the elements of $\Lambda$ with $\lambda_5>0$.  For these
elements, $\lambda_2=\lambda_3=0$.  Define a second homomorphism by
$$
x\longmapsto t,\qquad
y\longmapsto t^{b-d+1},\qquad
z\longmapsto t^{b-d+2}-\varepsilon t^{b-d+3}.
$$
For $\nu_2=2b-2d+3$, we have
\[
\begin{array}{c|c|c}
 \text{Element}&\psi_\varepsilon&\ord_\varepsilon\\ \hline
 f=y^2&t^{\nu_2-1}&\nu_2-1\\
 g_1&\varepsilon t^{\nu_2}&\nu_2\\
 g_4&t^{\nu_2}-\varepsilon t^{\nu_2+1}&\nu_2\\
 g_5&t^{\nu_2}-\varepsilon t^{\nu_2+1}&\nu_2
\end{array}
\]
For the omitted generators, the orders exceed $\nu_2$ by
\[
 \ord_\varepsilon(g_2)-\nu_2=2d-2,\qquad
 \ord_\varepsilon(g_3)-\nu_2=d-1,\qquad
 \ord_\varepsilon(g_6)-\nu_2=1.
\]
Hence $\psi_\varepsilon(I)k[[t]]\subseteq(t^{
\nu_2})$. For this second map, set
\[
 G_2=\{i\mid\ord_t\psi_\varepsilon(g_i)=\nu_2\}=\{1,4,5\}.
\]
 Every remaining exponent vector has support in
$G_2$, so the two subsets cover $\Lambda$. Apply
Lemma~\ref{lem:coefficient-separation} to the remaining sum.

For a fixed value $\lambda_1=i$, the exponents $\lambda_4$ and $\lambda_5$
satisfy
$$
\begin{cases}
 \lambda_4+\lambda_5=s-m-i,\\
 \lambda_4\deg g_4+\lambda_5\deg g_5
 =\delta-m\ell-i\deg g_1.
\end{cases}
$$
Since $\deg g_5-\deg g_4=3d-2\neq0$, the solution is unique if it exists.
The remaining coefficients therefore vanish.

\medskip

\noindent
\underline{$d\geq2$, $e=0$}:
Use first the map from the preceding case with $\nu_1=2b+1$.
Here $g_2$ has order $\nu_1+1$, and $g_6$ has order
$2b+2d>\nu_1$. Thus $\psi_\varepsilon(I)k[[t]]\subseteq(t^{\nu_1})$. Set
\[
 G_1=\{i\mid\ord_t\psi_\varepsilon(g_i)=\nu_1\}=\{1,3,4\}.
\]
The elements of $\Lambda$ with support in $G_1$ are precisely those with $\lambda_5=0$. We check injectivity of $\lambda\mapsto\lambda_1$
on this subset for Lemma~\ref{lem:coefficient-separation}.

After fixing $\lambda_1=i$, the remaining exponents satisfy
$$
\begin{cases}
 \lambda_3+\lambda_4=s-m-i,\\
 \lambda_3\deg g_3+\lambda_4\deg g_4
 =\delta-m\ell-i\deg g_1.
\end{cases}
$$
Here $\deg g_4-\deg g_3=2$, so there is at most one solution.  Hence all the
coefficients with $\lambda_5=0$ vanish.

Every remaining element of $\Lambda$ has $\lambda_5>0$ and therefore
$\lambda_3=0$. Apply the second map from the preceding case. Now
\[
 \ord_\varepsilon(g_2)-\nu_2=2d-1,\qquad
 \ord_\varepsilon(g_3)-\nu_2=d-1,\qquad
 \ord_\varepsilon(g_6)-\nu_2=1,
\]
so $\psi_\varepsilon(I)k[[t]]\subseteq(t^{\nu_2})$. For this second map, set
\[
 G_2=\{i\mid\ord_t\psi_\varepsilon(g_i)=\nu_2\}=\{1,4,5\}.
\]

Every remaining exponent vector has support in $G_2$, and the two
subsets therefore cover $\Lambda$. Apply Lemma~\ref{lem:coefficient-separation}
to the remaining sum. For fixed $\lambda_1=i$, the exponents
$\lambda_4$ and $\lambda_5$ satisfy the same
number-of-factors equation as above and the corresponding weighted-degree
equation.  Since $\deg g_5-\deg g_4=3d-1\neq0$, the solution is unique if it
exists.  Thus the remaining coefficients vanish as well.

In all six cases, every coefficient $c_\lambda$ is zero.  Hence $\omega=0$,
contrary to the minimality of $m$.  It follows that $I^s$ is integrally
closed for every $s\geq1$, and therefore $I$ is normal.

Combining this with Proposition~\ref{prop:family-one-local-reductions}, we obtain the following.

\begin{theorem}
\label{thm:family-one}
Let $I$ be an ideal in the family~\eqref{eq:family-one}.  Then the Rees algebra $\calR(I)$ is a Cohen--Macaulay normal domain.
\end{theorem}

\subsection{The second six-generated family}
\label{sec:family-two}

In this subsection, we treat the second family \eqref{eq:family-two} in the
Introduction.  Put
$$
\begin{aligned}
 I={}&(g_1=y^2-x^{2b+1},\ g_2=x^{2b+2},\ g_3=x^{b+1}y,\
        g_4=x^rz,\ g_5=yz,\ g_6=z^2),\\
 &\hspace{35mm} b\geq1,\qquad 1\leq r\leq b+1.
\end{aligned}
$$

\subsubsection{Cohen--Macaulayness}

\begin{proposition}
\label{prop:family-two-reductions}
Define $Q\subseteq I$ by
$$
 Q=
 \begin{cases}
  (g_1,g_3+g_6,g_5),&r=b+1,\\
  (g_1,g_3+g_6,g_4),&1\leq r\leq b.
 \end{cases}
$$
Then the stronger equality $I^2=QI$ holds.  Consequently, $\calR(I)$ is Cohen--Macaulay by
Proposition~\ref{prop:local-reduction-criterion}.
\end{proposition}

\begin{proof}
We consider the two cases in the definition of $Q$.

\medskip

\noindent
\underline{$r=b+1$}: Since $Q=(g_1,g_3+g_6,g_5)$, we have
$I=Q+(g_2,g_4,g_6)$.  Hence it is enough to treat the six products
$g_ig_j$ with $i,j\in\{2,4,6\}$.
\begin{itemize}
 \item $g_2^2=xg_3(g_3+g_6)-xg_4g_5-xg_1g_2\in QI$.
 \item $g_2g_4=xg_3g_5-xg_1g_4\in QI$.
 \item $g_2g_6=xg_5^2-xg_1g_6\in QI$.
 \item $g_4^2=g_2g_6\in QI$.
 \item $g_4g_6=(g_3+g_6)g_4-g_2g_5\in QI$.
 \item $g_6^2=(g_3+g_6)g_6-g_4g_5\in QI$.
\end{itemize}
Thus $I^2=QI$.

\medskip

\noindent
\underline{$1\leq r\leq b$}: Since $Q=(g_1,g_3+g_6,g_4)$, we have
$I=Q+(g_2,g_5,g_6)$.  The six remaining products are treated as follows:
\begin{itemize}
 \item
 $g_2^2=xg_3(g_3+g_6)-x^{b+2-r}g_4g_5-xg_1g_2\in QI$.
 \item $g_2g_5=x^{b+1-r}g_3g_4\in QI$.
 \item $g_2g_6=x^{2b+2-2r}g_4^2\in QI$.
 \item $g_5^2=g_1g_6+x^{2b+1-2r}g_4^2\in QI$.
 \item
 $g_5g_6=(g_3+g_6)g_5-x^{b+1-r}g_1g_4-x^{b-r}g_2g_4\in QI$.
 \item $g_6^2=(g_3+g_6)g_6-x^{b+1-r}g_4g_5\in QI$.
\end{itemize}
Thus $I^2=QI$ also in this case.  In particular, unlike the calculation for
the first family, these choices of $Q$ are reductions of $I$ already in $S$.
\end{proof}

\subsubsection{Normality}

We next prove that $I$ is normal.  Recall from
Subsection~\ref{subsec:symmetric-multiplicity-four} that we have chosen the
grading given by $\deg x=4$, $\deg y=4b+2$, and
$\deg z=8b+7-4r$.  For this grading, the least relation degree is
$\ell=8b+4$.  The ideal $I$ itself is homogeneous for any positive choice of
$\deg z$.

Let $f=y^2$ and $J=I+(f)$.  We have
$$
 J=(y^2,x^{2b+1},x^{b+1}y,x^rz,yz,z^2).
$$
As observed in Section~\ref{sec:target-ideals}, $J=I_\ell$ and hence $J$
is integrally closed.  Since $J$ is a monomial ideal with six minimal
generators, it is normal by \cite{AtakaMatsuoka2026}.

\begin{lemma}\label{lem:common_secondfamily}
The following assertions hold.
\begin{enumerate}
 \item $\m f\subseteq I$.
 \item $fg_2,fg_6\in I^2$.
 \item If $r=b+1$, then $fg_4\in I^2$.
 \item If $1\leq r\leq b$, then $fg_5\in I^2$.
\end{enumerate}
\end{lemma}

\begin{proof}
The first assertion follows from
$$
 xf=xg_1+g_2,\qquad
 yf=yg_1+x^bg_3,\qquad
 zf=zg_1+x^{2b+1-r}g_4.
$$
The remaining assertions follow from
$$
 fg_2=g_3^2,\qquad fg_6=g_5^2,
$$
together with
$$
 fg_4=g_3g_5\quad(r=b+1),\qquad
 fg_5=g_1g_5+x^{b-r}g_3g_4\quad(1\leq r\leq b).
$$
\end{proof}

We apply the strategy described in Subsection~\ref{subsec:normality}.
Suppose that $I^s$ is not integrally closed for some positive integer $s$.
Choose $\xi$, $m$, and $\omega$ as in that subsection, so that
$$
 \xi=\xi_0+f^m\omega,
$$
where
$\xi_0\in I^s+fI^{s-1}+\cdots+f^{m-1}I^{s-m+1}$ and
$\omega\in I^{s-m}$.  By Lemma~\ref{lem:common_secondfamily}, we may write
$$
 \omega=\sum_{\lambda\in\Lambda}c_\lambda g^\lambda,
 \qquad c_\lambda\in k,
$$
where
$$
 \Lambda=
 \begin{cases}
 \{\lambda\in\Lambda^{s-m}_{\delta-m\ell}
       \mid\lambda_2=\lambda_4=\lambda_6=0\},&r=b+1,\\
 \{\lambda\in\Lambda^{s-m}_{\delta-m\ell}
       \mid\lambda_2=\lambda_5=\lambda_6=0\},&1\leq r\leq b.
 \end{cases}
$$

Put $\nu=4b+3$, and take $\varepsilon_1,\varepsilon_2\in k^\times$.  If
$\ch k\neq2$, define a homomorphism
$\psi_{\varepsilon_1,\varepsilon_2}\colon S\to k[[t]]$ by
$$
 x\longmapsto t^2,\qquad
 y\longmapsto t^{2b+1}+\varepsilon_1t^{2b+2}.
$$
If $\ch k=2$, use instead
$$
 x\longmapsto t^2(1+\varepsilon_1t),\qquad
 y\longmapsto t^{2b+1}.
$$
In either case, complete the definition by setting
$$
 z\longmapsto
 \begin{cases}
  \varepsilon_2t^{2b+2},&r=b+1,\\
  \varepsilon_2t^{\nu-2r},&1\leq r\leq b.
 \end{cases}
$$
Let $\kappa=2$ if $\ch k\neq2$, and let $\kappa=1$ if
$\ch k=2$.  For the elements that occur in $\omega$, the
leading terms are as follows:
\[
\begin{array}{c|c|c}
 \text{Element}&\psi_{\varepsilon_1,\varepsilon_2}&
 \ord_t\psi_{\varepsilon_1,\varepsilon_2}\\ \hline
 f=y^2&t^{\nu-1}+\text{(higher-order terms)}&\nu-1\\
 g_1&\kappa\varepsilon_1t^\nu+\text{(higher-order terms)}&\nu\\
 g_3&t^\nu+\text{(higher-order terms)}&\nu\\
 g_5\quad(r=b+1)&\varepsilon_2t^\nu+\text{(higher-order terms)}&\nu\\
 g_4\quad(1\leq r\leq b)&\varepsilon_2t^\nu+\text{(higher-order terms)}&\nu
\end{array}
\]
Moreover, $\psi_{\varepsilon_1,\varepsilon_2}(I)k[[t]]\subseteq(t^\nu)$.

Suppose first that $r=b+1$. The omitted generators $g_2,g_4,g_6$
have orders $\nu+1,\nu+1,\nu+1$, respectively. Set
\[
 G=\{i\mid\ord_t\psi_{\varepsilon_1,\varepsilon_2}(g_i)=\nu\}=\{1,3,5\}.
\]
Every $\lambda\in\Lambda$ has support in $G$.
In Lemma~\ref{lem:coefficient-separation}, the parameter exponent map is
$\lambda\mapsto(\lambda_1,\lambda_5)$.

The parameter exponents determine $\lambda_1$ and $\lambda_5$, respectively, and the
number-of-factors equation
$$
 \lambda_1+\lambda_3+\lambda_5=s-m
$$
then determines $\lambda_3$.  Thus distinct elements of $\Lambda$ give
distinct monomials in $\varepsilon_1$ and $\varepsilon_2$, and hence
$c_\lambda=0$ for every
$\lambda\in\Lambda$.

If $1\leq r\leq b$, the omitted generators $g_2,g_5,g_6$ have orders
$\nu+1$, $\nu+2b+1-2r$, and $\nu+4b+3-4r$, respectively.
Set
\[
 G=\{i\mid\ord_t\psi_{\varepsilon_1,\varepsilon_2}(g_i)=\nu\}=\{1,3,4\}.
\]
Every $\lambda\in\Lambda$ has support in $G$. Apply Lemma~\ref{lem:coefficient-separation}
with parameter exponent map $\lambda\mapsto(\lambda_1,\lambda_4)$.

Here the exponents of $\varepsilon_1$ and $\varepsilon_2$ determine
$\lambda_1$ and $\lambda_4$, while
$$
 \lambda_1+\lambda_3+\lambda_4=s-m
$$
determines $\lambda_3$.  Therefore every coefficient $c_\lambda$ is again
zero.

In both cases, $\omega=0$, contrary to the minimality of $m$.  It follows
that $I^s$ is integrally closed for every $s\geq1$, and hence $I$ is normal.

Combining this with Proposition~\ref{prop:family-two-reductions}, we obtain
the following.

\begin{theorem}
\label{thm:family-two}
Let $I$ be an ideal in the family~\eqref{eq:family-two}.  Then the Rees
algebra $\calR(I)$ is a Cohen--Macaulay normal domain.
\end{theorem}

\section{A seven-generated example}
\label{sec:nonsymmetric-example}

We now consider the non-symmetric numerical semigroup
$$
 H=\langle4,9,15\rangle.
$$
In the notation used for the case $v<2u$ in
Subsection~\ref{subsec:nonsymmetric-multiplicity-four}, this example
corresponds to $(b,c,e)=(6,4,0)$.
Relabel $x_1,x_2,x_3$ as $x,y,z$, respectively, so that
$\deg x=4$, $\deg y=9$, and $\deg z=15$.
Its Ap\'ery set with respect to $4$ is $\{0,9,15,18\}$, and the defining
ideal of $k[H]$ is
$$
 \fkp_H=(x^6-yz,\ y^3-x^3z,\ z^2-x^3y^2).
$$
The degrees of these three relations are $24$, $27$, and $30$, respectively.
Thus the least relation degree is $\ell=24$.  The least elements of $H$ that
are at least $25$ in the four residue classes modulo $4$ are represented by
$$
 x^7,\qquad x^4y,\qquad x^3z,\qquad x^2y^2.
$$
It follows that
$$
 I=I_{25}=(g_1=yz-x^6,\ g_2=x^7,\ g_3=x^4y,\ g_4=x^2y^2,
             \ g_5=x^3z,\ g_6=y^3,\ g_7=z^2).
$$
Indeed, the last two generators are obtained from the relations
$y^3-x^3z$ and $z^2-x^3y^2$.  By
Proposition~\ref{prop:inverse-image-ideals}(1), the ideal $I$ is integrally
closed, and Theorem~\ref{thm:classification}(3) shows that these seven
generators are minimal.

\subsection{Cohen--Macaulayness}

\begin{proposition}
\label{prop:example-local-reduction}
Let
$$
 Q=(g_1,g_4+g_5,g_3+g_6+g_7).
$$
Then
$$
 I^2=QI+\m I^2.
$$
Consequently, the reduction number of $IS_\m$ is exactly one, and
$\calR(I)$ is Cohen--Macaulay.
\end{proposition}

\begin{proof}
Since $I=Q+(g_2,g_3,g_5,g_7)$, it is enough to treat the ten products
$g_ig_j$ with $i,j\in\{2,3,5,7\}$.  We first note that
$g_2g_6=xg_3g_4$, $x^8yz=xg_3g_5$, and
$g_2g_4=xg_3^2$ all belong to $\m I^2$.  Moreover,
$$
 g_5g_6=x^3y^2g_1+g_2g_4\in QI+\m I^2.
$$
The remaining products are treated in the following order.
\begin{itemize}
 \item $g_2g_7=xg_5^2\in\m I^2$.
 \item
 $g_2g_3=(g_3+g_6+g_7)g_2-g_2g_6-g_2g_7\in QI+\m I^2$.
 \item
 $g_5^2=(g_4+g_5)g_5-x^5yg_1-g_2g_3\in QI+\m I^2$.
 \item
 $g_5g_7=(g_4+g_5)g_7-x^2yzg_1-x^8yz\in QI+\m I^2$.
 \item
 $g_3g_5=(g_3+g_6+g_7)g_5-g_5g_6-g_5g_7
 \in QI+\m I^2$.
 \item
 $g_2g_5=(g_4+g_5)g_2-g_2g_4\in QI+\m I^2$.
 \item
 $g_3g_7=x^4zg_1+g_2g_5\in QI+\m I^2$.
 \item
 Since $g_4g_5=x^5yg_1+g_2g_3$, we have
 $g_4^2=(g_4+g_5)g_4-g_4g_5\in QI+\m I^2$.  Hence
 $g_3^2=(g_3+g_6+g_7)g_3-g_4^2-g_3g_7\in QI+\m I^2$.
 \item
 We have $g_6g_7=y^2zg_1+xg_4g_5\in QI+\m I^2$.  Therefore
 $g_7^2=(g_3+g_6+g_7)g_7-g_3g_7-g_6g_7\in QI+\m I^2$.
 \item $g_2^2=xg_3g_5-xg_1g_2\in QI+\m I^2$.
\end{itemize}
Thus $I^2=QI+\m I^2$.

Since $I$ is $\m$-primary, localization at $\m$ preserves the number of
minimal generators.  Hence $IS_\m$ is not a parameter ideal, and its
reduction number cannot be zero.  The equality just proved shows that its
reduction number is one.  Proposition~\ref{prop:local-reduction-criterion}
now shows that $\calR(I)$ is Cohen--Macaulay.
\end{proof}

\subsection{Normality}

\begin{proposition}
\label{prop:example-normality}
The ideal $I$ is normal.
\end{proposition}

\begin{proof}
Set $f=yz$ and $J=I+(f)$.  Then
$$
 J=(yz,x^6,x^4y,x^2y^2,x^3z,y^3,z^2)=I_{24}.
$$
Thus $J$ is an integrally closed monomial ideal generated by at most seven
elements, and hence it is normal by \cite{AtakaMatsuoka2026}.  Moreover,
$$
 xf=xg_1+g_2,\qquad
 yf=yg_1+x^2g_3,\qquad
 zf=zg_1+x^3g_5,
$$
so $\m f\subseteq I$.  We shall also use the identities
$$
 fg_2=g_3g_5,\qquad
 fg_4=g_1g_4+g_3^2,\qquad
 fg_6=g_1g_6+g_3g_4,\qquad
 fg_7=g_1g_7+g_5^2.
$$

We use the setup of Subsection~\ref{subsec:normality} with $\mu=7$
and apply Lemma~\ref{lem:coefficient-separation}.

Suppose that $I^s$ is not integrally closed for some positive integer $s$.
Choose $\xi$, $m$, and $\omega$ as in that subsection, so that
$$
 \xi=\xi_0+f^m\omega,
$$
where
$\xi_0\in I^s+fI^{s-1}+\cdots+f^{m-1}I^{s-m+1}$ and
$\omega\in I^{s-m}$ is homogeneous of degree $\delta-24m$.
The four identities above allow us to remove every term involving
$g_2$, $g_4$, $g_6$, or $g_7$.  We may therefore write
$$
 \omega=\sum_{\lambda\in\Lambda}c_\lambda g^\lambda,
 \qquad c_\lambda\in k,
$$
where $g^\lambda=g_1^{\lambda_1}g_2^{\lambda_2}\cdots g_7^{\lambda_7}$ and
$$
\begin{aligned}
 \Lambda=\{\lambda=(\lambda_1,\lambda_2,\ldots,\lambda_7)\in\mathbb N^7
 \mid {}&\lambda_1+\lambda_2+\cdots+\lambda_7=s-m,\\
 &\lambda_1\deg g_1+\lambda_2\deg g_2+\cdots+\lambda_7\deg g_7
       =\delta-24m,\\
 &\lambda_2=\lambda_4=\lambda_6=\lambda_7=0\}.
\end{aligned}
$$

For $\varepsilon_1,\varepsilon_2\in k^\times$, define a homomorphism
$\psi_{\varepsilon_1,\varepsilon_2}\colon S\to k[[t]]$ by
$$
 x\longmapsto t^2,\qquad
 y\longmapsto\varepsilon_2t^5,\qquad
 z\longmapsto\varepsilon_2^{-1}t^7(1+\varepsilon_1t).
$$
Put $\nu=13$.  The elements occurring in $\omega$ have the following
images:
\[
\begin{array}{c|c|c}
 \text{Element}&\psi_{\varepsilon_1,\varepsilon_2}&
 \ord_t\psi_{\varepsilon_1,\varepsilon_2}\\ \hline
 f&t^{12}(1+\varepsilon_1t)&\nu-1\\
 g_1&\varepsilon_1t^{13}&\nu\\
 g_3&\varepsilon_2t^{13}&\nu\\
 g_5&\varepsilon_2^{-1}t^{13}(1+\varepsilon_1t)&\nu
\end{array}
\]
Moreover,
$$
 \ord_t\psi_{\varepsilon_1,\varepsilon_2}(g_2)
 =\ord_t\psi_{\varepsilon_1,\varepsilon_2}(g_4)
 =\ord_t\psi_{\varepsilon_1,\varepsilon_2}(g_7)=14,
 \qquad \ord_t\psi_{\varepsilon_1,\varepsilon_2}(g_6)=15.
$$
Hence
$\psi_{\varepsilon_1,\varepsilon_2}(I)k[[t]]\subseteq(t^\nu)$.

Set
\[
 G=\{i\mid\ord_t\psi_{\varepsilon_1,\varepsilon_2}(g_i)=\nu\}=\{1,3,5\}.
\]
Every $\lambda\in\Lambda$ has support in $G$. The parameter exponent
map in Lemma~\ref{lem:coefficient-separation} is
$\lambda\mapsto(\lambda_1,\lambda_3-\lambda_5)$.

It remains to check that distinct elements of $\Lambda$ give distinct
parameter monomials.  The exponent of $\varepsilon_1$ determines $\lambda_1$,
and the exponent of $\varepsilon_2$ determines $\lambda_3-\lambda_5$.  On the other hand,
the number of factors gives
$$
 \lambda_3+\lambda_5=s-m-\lambda_1.
$$
These two integer equations determine $\lambda_3$ and $\lambda_5$ uniquely
whenever a solution exists.  Thus every $c_\lambda$ is zero.  This gives
$\omega=0$, contrary to the minimality of $m$.  It follows that $I^s$ is
integrally closed for every $s\geq1$, and hence $I$ is normal.
\end{proof}

Combining Propositions~\ref{prop:example-local-reduction}
and~\ref{prop:example-normality}, we obtain the following.

\begin{theorem}
\label{thm:example-rees}
For the ideal $I=I_{25}$ arising from $H=\langle4,9,15\rangle$, the Rees
algebra $\calR(I)$ is a Cohen--Macaulay normal domain.
\end{theorem}

\section{Further questions}
\label{sec:further-questions}

We conclude with several questions that are not settled in this paper.
Whenever $H$ is a numerical semigroup considered below, $\ell$ denotes the
least weighted degree of a nonzero homogeneous element of the defining ideal
$\fkp_H$ of $k[H]$.

\begin{question}
\label{que:seven-from-semigroups}
	Are all the seven-generated ideals $I_{\ell+1}$ arising
	from non-symmetric numerical semigroups of multiplicity four and embedding
	dimension three normal?  Is the same true for those arising from symmetric
	numerical semigroups of multiplicity five and embedding dimension three?
\end{question}

For the next question, let $H$ be a numerical
semigroup of embedding dimension three and arbitrary multiplicity.  Since
$I_\ell$ is an integrally closed monomial ideal, it is normal by
\cite[Theorem~3.1]{AtakaMatsuoka2026} whenever $\mu_S(I_\ell)\leq7$.

\begin{question}
	Suppose that $\mu_S(I_\ell)\leq7$ and
	$\mu_S(I_{\ell+1})\leq7$.  Must $I_{\ell+1}$ be normal?
\end{question}

The preceding question concerns only the passage from $I_\ell$ to
$I_{\ell+1}$.  The following formulation asks the same at every step.
Let $H$ be a numerical semigroup of embedding dimension three, and retain
the notation $S$, $\varphi_H$, and
\[
 I_h=\varphi_H^{-1}\bigl(t^hk[t]\cap k[H]\bigr)
\]
of Section~\ref{sec:target-ideals}.

\begin{question}
\label{que:successive-degree-bounds}
For $h\in H$, set
\[
 \sigma(h)=\min\{n\in H\mid n>h\}.
\]
Suppose that $\mu_S(I_h)\leq7$ and $\mu_S(I_{\sigma(h)})\leq7$.
If $I_h$ is normal, must $I_{\sigma(h)}$ be normal?
\end{question}

The following example shows that normality need not pass to the next
cutoff without restrictions on the number of generators, even when both
ideals are monomial.

\begin{example}\label{ex:successive-cutoffs-eight-generators}
Let $H=\langle6,14,21\rangle$, and give $S=k[x,y,z]$ the grading
$\deg x=6$, $\deg y=14$, and $\deg z=21$.  The defining ideal is
\[
 \fkp_H=(y^3-x^7,\ z^2-x^7),
\]
so $\ell=42$.  Since $41=6+14+21\in H$ and $42\in H$, we have
$\sigma(41)=42$.  The two successive ideals are
\[
\begin{aligned}
 I_{41}&=(x^7,x^5y,x^3y^2,y^3,x^4z,xyz,y^2z,z^2),\\
 I_{42}&=(x^7,x^5y,x^3y^2,y^3,x^4z,x^2yz,y^2z,z^2).
\end{aligned}
\]
Both are integrally closed monomial ideals with eight minimal generators.
A direct computation with the Newton polyhedron of $I_{41}^2$ shows
that $I_{41}^2$ is integrally closed.  Hence $I_{41}$ is normal by
\cite[Proposition~3.1]{ReidRobertsVitulli2003}.

On the other hand, $I_{42}=\ol{(x^7,y^3,z^2)}$ is not normal.
Indeed, put $w=x^6y^2z$.  Then $w\notin I_{42}^2$, whereas
\[
 w^2=(x^7)(x^5y)(y^3)(z^2)\in I_{42}^4.
\]
Thus $w\in\ol{I_{42}^2}\setminus I_{42}^2$.  In contrast,
$w=(xyz)(x^5y)\in I_{41}^2$.
\end{example}

\section*{Acknowledgments}
This work grew out of discussions with the late Professor Shiro Goto.
Together, we analyzed a particular example whose behavior led me to the
general results developed in this paper.

\bibliographystyle{amsplain}
\bibliography{refs}

\end{document}